\documentclass[reqno,11pt]{amsart}

\usepackage{hyperref}
\usepackage{a4wide,amsthm,amssymb,mathrsfs,setspace,pstricks,booktabs,mathtools,amsmath,amsfonts}
\usepackage{graphicx}
\usepackage{xcolor}
\usepackage{array}
\newtheorem{theorem}{Theorem}[section]
\newtheorem{lemma}[theorem]{Lemma}

\newtheorem{corollary}[theorem]{Corollary}

\theoremstyle{definition}

\newtheorem{theoremA}{Theorem}

\newtheorem{theoremB}{Theorem}

\theoremstyle{remark}
\newtheorem{remark}[theorem]{Remark}

\numberwithin{equation}{section}

\newcommand{\R}{\mathbb{R}}
\newcommand{\C}{\mathbb{C}}

\begin{document}
	
	\title [Curvature equations with two singular sources on rectangular torus]{Sharp criteria for the existence of solutions to curvature equations with two singular sources on rectangular torus}
	
	\author{Zhijie Chen}
	\address{Department of Mathematical Sciences, Yau Mathematical Sciences Center, Tsinghua University, Beijing, 100084, China}
	\email{zjchen2016@tsinghua.edu.cn}

	\author{Zhen Song}    
	\address{Department of Mathematics, Northeastern University, Shenyang 110004, China}
	\email{songzhen@mail.neu.edu.cn}

	\subjclass[2020]{Primary 35J75; Secondary 33E05; 34M35; 35J08.}

	\keywords{Curvature equation; Rectangular torus; Generalized Lam\'e equation.}
	
	\begin{abstract}
		We study the following curvature equation with two singular sources on a flat torus
		\begin{equation*}
			\Delta u+e^u=8\pi(\delta_p+\delta_{-p})\quad\text{on }E_\tau\coloneqq\C/(\mathbb{Z}+\mathbb{Z}\tau),
		\end{equation*}
	where $\delta_p$ denotes the Dirac measure at $p$. The case $\wp'(2p)=0$ was solved by Kuo (J. Differential Geom. 2026).
In this paper, we study the general case $\wp'(2p)\neq 0$. A novel phenomenon for $\wp'(2p)\neq 0$ is that the Hill discriminant of the associated generalized Lam\'{e} equation has an essential singularity.   We develop new ideas to determine the intersection of two conditional stability sets completely, which is quite surprising since it seems impossible to determine the exact structure of these two sets due to the essential singularity. This leads to a sharp criteria for the existence and non-existence of solutions when $E_\tau$ is a rectangular torus and $\wp(p)\in \R$.
	\end{abstract}
	
	\maketitle
	
	\section{Introduction}
	
	Let $\tau\in\mathbb{H}\coloneqq\{\tau\in \mathbb{C}:\operatorname{Im}\tau>0\}$, $\Lambda_\tau\coloneqq\mathbb{Z}+\mathbb{Z}\tau$, and denote
	\begin{equation}
		\omega_0=0,\quad\omega_1=1,\quad\omega_2=\tau,\quad\omega_3=1+\tau.
	\end{equation}
	Let $E_\tau\coloneqq\mathbb{C}/\Lambda_\tau$ be a flat torus in the plane and $E_\tau[2]\coloneqq\{\frac{\omega_k}{2}:k=0,1,2,3\}+\Lambda_\tau$ be the set consisting of the lattice points and half periods in $E_\tau$. 
	Let $\wp(z)=\wp(z;\tau)$ be the Weierstrass $\wp$-function with periods $\Lambda_\tau$, defined by
	\begin{equation}
		\wp(z)\coloneqq\frac{1}{z^2}+\sum_{\omega\in\Lambda_\tau\setminus\{0\}}\left(\frac{1}{(z-\omega)^2}-\frac{1}{\omega^2}\right),
	\end{equation}
	which satisfies the well-known cubic equation
	\begin{equation}\label{eq:cubic}
		\wp'(z)^{2}=4\wp(z)^{3}-g_{2}\wp(z)-g_{3}=4\prod_{k=1}^3(\wp(z)-e_k),
	\end{equation}
	where $g_{2}=g_2(\tau)$ and $g_{3}=g_3(\tau)$ are known as invariants of the elliptic curve $E_\tau$, and $e_k=e_k(\tau)\coloneqq\wp(\frac{\omega_k}{2})$ for $k=1,2,3$.   Define the Weierstrass zeta function
	\begin{equation}
		\zeta(z)=\zeta(z;\tau)\coloneqq-\int^z\wp(\xi;\tau)\,d\xi.
	\end{equation}
It is well-known that $\zeta(z)$ is an odd meromorphic function with simple poles at the lattice $\Lambda_\tau$, and has two quasi-periods $\eta_k=\eta_k(\tau)$, $k=1,2$, defined by
\begin{equation}
	\zeta(z+\omega_1)=\zeta(z)+\eta_1\quad\text{and}\quad\zeta(z+\omega_2)=\zeta(z)+\eta_2.
\end{equation}
	
	Consider the following curvature equation
	\begin{equation}\label{main problem general case}
		\Delta u+e^u=4\pi n(\delta_p+\delta_{-p})\quad\text{on }E_\tau,
	\end{equation}
	where $n\in\mathbb{N}_+$, $p\in E_\tau$ and $\delta_p$ denotes the Dirac measure at $p$. Equation \eqref{main problem general case} is also known as a singular Liouville equation, and arises from both conformal geometry and mathematical physics. Geometrically, a solution $u$ of \eqref{main problem general case} leads to a spherical metric $ds^2=\frac{1}{2}e^u|dz|^2$ with constant Gaussian curvature $+1$ acquiring conic singularities at $\pm p$. In statistical physics, \eqref{main problem general case} is a special case of the general mean field equation that appears as the equation for the mean field limit of the Euler flow in Onsager's vortex model.  We refer the readers to \cite{BartolucciJevnikarLeeYang2019,BartolucciTarantello2002,	BartolucciYangZhang2026,CM,ChaiLinWang2015,ChenLin2003,DKM,FengSongXu2026,
FuLin2026,GuiMoradifam2018,Kuo2026,LinWang2010,LinYan2013,LuoTian1992,MalchiodiRuiz2011,
WeiZhang2026,WeiWuXu2026} and the references therein for introductions to mean field equations and related topics, as well as for recent developments in these areas.
	
	From the view point of PDEs, since the a priori estimate can not hold for \eqref{main problem general case} (see \cite{BartolucciTarantello2002}), many powerful PDE approaches, such as variational methods, topological degree methods and so on, seems not work for \eqref{main problem general case}, which makes the study of  the existence  of solutions for \eqref{main problem general case} quite challenging.

Starting from the seminal work by Lin and Wang \cite{LinWang2010},	
the case $p\in E_\tau[2]$, i.e., $p=-p$ in $E_{\tau}$, has been widely studied in \cite{ChaiLinWang2015, ChenFuLin2021, ChenLin2020, Eremenko, EG, EMP, Lin2026, LinWang2010, LinWangJEP}.
In this case,  up to a translation,  equation \eqref{main problem general case} becomes
	\begin{equation}\label{alpha=1,p half period}
		\Delta u+e^u=8\pi n\delta_{0}\quad\text{on }E_\tau.
	\end{equation}
	It turns out that whether solutions of \eqref{alpha=1,p half period} exist or not depends essentially on the choice of $\tau$, a novel phenomenon first discovered by Lin and Wang \cite{LinWang2010}, where they studied the ``simplest'' case $n=1$. 
	Let $G(z)=G(z;\tau)$ be the Green function on $E_\tau$ defined by
	\begin{equation}
		-\Delta G(z)=\delta_0-\frac{1}{|E_\tau|},\quad\int_{E_\tau}G(z)\,dxdy=0,
	\end{equation}
	where $|E_\tau|$ is the area of the torus $E_\tau$. It is an even function with the only singularity at $0$, so the three half-period points $\frac{\omega_k}{2}$, $k=1,2,3$, are always critical points of $G(z)$, and other critical points (if exist) must appear in pairs.	
Lin and Wang \cite{LinWang2010} proved that the Green function $G(z)$ has either three or five critical points, and \eqref{alpha=1,p half period} with $n=1$ has solutions if and only if $G(z)$ has five critical points.
For example, when $\tau\in i\R_{>0}$, i.e., $E_\tau$ is a rectangular torus, $G(z)$ has only three critical points, or equivalently, \eqref{alpha=1,p half period} with $n=1$ has no solutions. When $\tau=\frac{1}{2}+ib$, i.e., $E_\tau$ is a rhombus torus, there exist $0<b_0<\frac{1}{2}<b_1$ such that $G(z)$ has five critical points, or equivalently,  \eqref{alpha=1,p half period} with $n=1$ has solutions, if and only if $b\in (0,b_0)\cup(b_1,+\infty)$. For general $n\geq 2$, Chai, Lin and Wang \cite{ChaiLinWang2015, LinWangJEP} studied \eqref{alpha=1,p half period} by the method of algebraic geometry, and established the deep connection with the classical Lam\'{e} equation and modular forms. In particular, they proved that the following statement holds for \eqref{alpha=1,p half period}:
\begin{equation} \label{eq: data}
  \parbox{\dimexpr\linewidth-5em}{
  Once there is a solution, then there exists a one-parameter scaling family of solutions $u_{\beta}(z)$, where $\beta>0$ is arbitrary. Furthermore, $u_\beta(z)$ is even (i.e., $u_{\beta}(z)=u_\beta(-z)$) if and only if $\beta=1$.
  }
\end{equation}
Later, Chen and Lin \cite{ChenLin2020} proved that when $\tau\in i\R_{>0}$, i.e., $E_\tau$ is a rectangular torus, \eqref{alpha=1,p half period} has no solutions for any $n\in\mathbb N_+$.

For the general case $p\in E_\tau\backslash E_\tau[2]$, i.e., $p\neq -p$ in $E_{\tau}$, equation \eqref{main problem general case} has two singular sources and is more difficult to handle, and as far as we know, there are only several works \cite{ChenFuLin2025, ChenFuLin2026, ChenFuLinSong2026, Kuo2026} studying it. 
Remark that, by following the approach of \cite{ChaiLinWang2015}, it is easy to prove that the statement \eqref{eq: data} also holds for \eqref{main problem general case}; see e.g. Section \ref{section 3} for a detailed proof for $n=2$.

More precisely, inspired by \cite{LinWang2010},
the series of papers \cite{ChenFuLin2025, ChenFuLin2026, ChenFuLinSong2026} studied  \eqref{main problem general case} with $n=1$, i.e.,
	\begin{equation}\label{alpha=1,p not half period}
		\Delta u+e^u=4\pi(\delta_p+\delta_{-p})\quad\text{on }E_\tau.
	\end{equation}	
By establishing the deep connection between \eqref{alpha=1,p not half period}, the sum of two Green functions
	\begin{equation*}
		G_p(z)\coloneqq\frac12(G(z+p)+G(z-p)),
	\end{equation*}
	and Hitchin's formula \cite{Hitchin1995} arising from his study of Einstein metric, 
it was proved in \cite{ChenFuLin2025} that \eqref{alpha=1,p not half period} has at most three one-parameter scaling families of solutions, and every number in $\{0, 1, 2, 3\}$ really occurs for different choices of $(\tau, p)$'s.

However, given a specific torus and $p\in E_\tau\backslash E_\tau[2]$, it seems very difficult to determine explicitly the number of solution families for \eqref{alpha=1,p not half period} in general. It is well known that $\wp(\cdot): E_{\tau}\to \mathbb{C}\cup\{\infty\}$ is a double cover with branch points at $\{\frac{\omega_k}{2}\}_{k=0}^3$, i.e., for any $c\in \mathbb{C}\cup\{\infty\}$, there is a unique pair $\pm z_c\in E_{\tau}$ such that $\wp(\pm z_{c})=c$. Thus $$\text{there is a one-to-one correspondence between $\pm p\in E_{\tau}$ and $\wp(p)\in\mathbb C\cup\{\infty\}.$}$$ 
The subsequent paper \cite{ChenFuLin2026} studied the important case $\tau=ib$ with $b>0$, i.e., $E_\tau$ is a rectangular torus, and gave a complete answer for $\wp(p)\in\R$ as follows.
	\begin{theoremA}[\cite{ChenFuLin2026}]\label{theorem c}
		\emph{Let $\tau=ib$ with $b>0$. Then there are $8$ real values, denoted by $d_1<\cdots<d_8$, such that the following statements hold.
		\begin{itemize}
			\item  Equation \eqref{alpha=1,p not half period} has no solutions for
			$$
			\wp(p)\in(-\infty,d_1]\cup[d_2,d_3]\cup[d_4,d_5]\cup[d_6,d_7]\cup[d_8,+\infty).
			$$
			\item Equation \eqref{alpha=1,p not half period} has a unique one-parameter scaling family of solutions for
			$$
			\wp(p)\in(d_1,d_2)\cup(d_3,d_4)\cup(d_5,d_6)\cup(d_7,d_8).
			$$
		\end{itemize}}
	\end{theoremA}

See also \cite{ChenFuLinSong2026} for the study of \eqref{alpha=1,p not half period} under the other important case $\tau=\frac{1}{2}+ib$ with $b>0$, i.e., $E_\tau$ is a rhombus torus. In contrast to the rectangular case, the rhombus torus case exhibits several different phenomena arising from geometric differences, and the results depends heavily on the value of $b$.
	
Equation \eqref{main problem general case} with $n\geq 2$ seems too difficult to be studied in a unified way. For $n=2$, i.e.,
	\begin{equation}\label{main problem}
		\Delta u+e^u=8\pi(\delta_p+\delta_{-p})\quad\text{on }E_\tau,
	\end{equation}
	recently,  Kuo \cite{Kuo2026}
studied the special case $\wp'(2p)=0$, i.e., $2p=\frac{\omega_k}{2}$, $k=1,2,3$, is a half period. In this case, up to a translation, equation \eqref{main problem} becomes
	\begin{equation}\label{alpha=2,2p half period}
		\Delta u+e^u=8\pi(\delta_0+\delta_{\frac{\omega_k}{2}})\quad\text{on }E_\tau.
	\end{equation}
	In contrast to \eqref{main problem general case} for which the statement \eqref{eq: data} holds,
Kuo \cite{Kuo2026} discovered that equation \eqref{alpha=2,2p half period} could have a one-parameter scaling family of solutions $u_\beta(z)$ such that $u_\beta(z)\neq u_{\beta}(-z)$ for any $\beta$ for some torus. Such family of solutions was called a non-even family of solutions in \cite{Kuo2026}. Specifically, when $\tau=ib$ with $b>0$, i.e., $E_\tau$ is a rectangular torus, Kuo's result for \eqref{alpha=2,2p half period} implies the following result for \eqref{main problem}.

	\begin{theoremB}[\cite{Kuo2026}]\label{theorem d}
		\emph{Let $\tau=ib$ with $b>0$. Then equation \eqref{main problem} has no solutions for $2p=\frac{\omega_k}{2}$, $k=1,2$. However, if $2p=\frac{\omega_3}{2}$, then there exists $b_0>1$ such that equation \eqref{main problem} has solutions for $b\in (0, \frac1{b_0})\cup(b_0,+\infty)$.}
	\end{theoremB}
	
	Remark that for $\tau=ib$ with $b>0$, $\wp(p)\in\mathbb{R}$ if $2p=\frac{\omega_k}{2}$ for $k=1,2$, but $\wp(p)\notin\mathbb{R}$ if $2p=\frac{\omega_3}{2}$.
	To the best of our knowledge, there seems no any results for \eqref{main problem} when $p\in E_\tau\backslash E_\tau[2]$ satisfies
 $\wp'(2p)\neq0$. 

	In this paper, we study the curvature equation \eqref{main problem} under the assumption
\begin{equation}\label{eq:p-ass}p\in E_\tau\backslash E_\tau[2]\qquad\text{and}\qquad\wp'(2p)\neq 0.\end{equation}
 One will see in Section \ref{section 2} that the case $\wp'(2p)\neq 0$ is much more difficult than the case $\wp'(2p)=0$. Like \cite{ChenFuLin2026}, we would like to study the important case $\tau=ib$ with $b>0$, i.e., $E_\tau$ is a rectangular torus. By developing a quite different approach from \cite{ChenFuLin2026, Kuo2026}, we can give a complete answer for all $\wp(p)\in\R$. Recall that if $p\in E_\tau[2]$ or $2p=\frac{\omega_k}{2}$, $k=1,2$, then by \cite[Theorem 1.1]{ChenLin2020} and Theorem \ref{theorem d}, we know that \eqref{main problem} has no solutions.
	
	To state the main results of this paper, we need to fix some notations. When $\tau=ib$ with $b>0$ and $\wp(p)\in\R$, by \cite[Lemma 4.1]{ChenFuLin2026}, we can assume without loss of generality that
	\begin{equation}
		\begin{aligned}
			p&\in\left(0,\frac{1}{2}\right]\cup \left[\frac{1}{2},\frac{1+\tau}{2}\right]\cup \left[\frac{1+\tau}{2},\frac{\tau}{2}\right]\cup \left[\frac{\tau}{2},0\right).
		\end{aligned}
	\end{equation}
	Here, $[z_1,z_2]\coloneqq \{(1-t)z_1+tz_2:0\leq t\leq1\}$, and $[z_1,z_2)$, $(z_1,z_2]$, $(z_1,z_2)$ are defined similarly. Together with the assumption \eqref{eq:p-ass}, we may assume further that
	\begin{equation}
		\begin{aligned}\label{range of p}
			p&\in\left(0,\frac{1}{4}\right)\cup \left(\frac{1}{4},\frac{1}{2}\right)\cup \left(\frac{1}{2},\frac{1}{2}+\frac{\tau}{4}\right)\cup \left(\frac{1}{2}+\frac{\tau}{4},\frac{1}{2}+\frac{\tau}{2}\right)\\
			&\quad\cup \left(\frac{1}{2}+\frac{\tau}{2},\frac{1}{4}+\frac{\tau}{2}\right)\cup \left(\frac{1}{4}+\frac{\tau}{2},\frac{\tau}{2}\right)\cup \left(\frac{\tau}{2},\frac{\tau}{4}\right)\cup \left(\frac{\tau}{4},0\right).
		\end{aligned}
	\end{equation}
	Now, we can state the main results of this paper. We first consider the case where $p\in(0,1/4)$.
	\begin{theorem}\label{main result for 0,1/4}
		Let $\tau=ib$ with $b>0$, i.e., $E_\tau$ is a rectangular torus, and let $p\in(0,1/4)$. Then there exist $p_-(b)$ and $p_+(b)$ satisfying
		\begin{equation}\label{p- p+ compare}
			0<p_-(b)<\frac16<p_+(b)<\frac14,
		\end{equation}
		such that the following statements hold.
		\begin{itemize}
			\item[(1).] If $p\in(0,p_-(b)]\cup[p_+(b),\frac14)$, then the curvature equation \eqref{main problem} has no solutions.
			\item[(2).] If $p\in (p_-(b),p_+(b))$, then the curvature equation \eqref{main problem} has a unique one-parameter scaling family of solutions $u_\beta(z)$, where $\beta>0$ is
			arbitrary. Furthermore, $u_\beta(z)=u_\beta(-z)$ if and only if $\beta=1$.
		\end{itemize}
	\end{theorem}
	\begin{remark}
		The solution $u_1(z)$ obtained in Theorem \ref{main result for 0,1/4} satisfies $u_1(z)=u_1(\bar{z})$ by using the complex variable $z=x+iy$. Indeed, since $u_1(\bar{z})$ is also an even solution of \eqref{main problem},
		the uniqueness of even solutions (see Theorem \ref{solution = unitary}) implies $u_1(z)=u_1(\bar{z})$.
	\end{remark}
	The remaining cases in \eqref{range of p} can be treated by reflection, translation, and modular transformation. Note that by \eqref{p- p+ compare}, we know $(1/2-p_+(b),1/2-p_-(b))\subset(1/4,1/2)$. To present our results more intuitively and concisely, for every $b > 0$, define
	\begin{equation}\label{Jb}
		J(b)\coloneqq(p_-(b),p_+(b))\cup\left(\frac12-p_+(b),\frac12-p_-(b)\right).
	\end{equation}
	We obtain the following result.
	\begin{theorem}\label{main result general}
		Let $\tau=ib$ with $b>0$, i.e., $E_\tau$ is a rectangular torus, and let $p$ satisfy \eqref{range of p}. Then the following statements hold.
		\begin{itemize}
			\item[(1).] Suppose that $p=t$, where $t\in(0,\frac14)\cup(\frac14,\frac12)$, then,
			\begin{itemize}
				\item[(1-1).] the curvature equation \eqref{main problem} has no solutions if $t\not\in J(b)$;
				\item[(1-2).] the curvature equation \eqref{main problem} has a unique one-parameter scaling family of solutions $\{u_\beta(z)\}_{\beta>0}$ if $t\in J(b)$. 
			\end{itemize}
			\item[(2).] Suppose that $p=t+\frac{\tau}{2}$, where $t\in(0,\frac14)\cup(\frac14,\frac12)$, then,
			\begin{itemize}
				\item[(2-1).] the curvature equation \eqref{main problem} has no solutions if $t\not\in J(b)$;
				\item[(2-2).] the curvature equation \eqref{main problem} has a unique one-parameter scaling family of solutions $\{u_\beta(z)\}_{\beta>0}$ if $t\in J(b)$. 
			\end{itemize}
			\item[(3).] Suppose that $p=\tau t$, where $t\in(0,\frac14)\cup(\frac14,\frac12)$, then,
			\begin{itemize}
				\item[(3-1).] the curvature equation \eqref{main problem} has no solutions if $t\not\in J(\frac1b)$;
				\item[(3-2).] the curvature equation \eqref{main problem} has a unique one-parameter scaling family of solutions $\{u_\beta(z)\}_{\beta>0}$ if $t\in J(\frac1b)$. 
			\end{itemize}
			\item[(4).] Suppose that $p=\frac12+\tau t$, where $t\in(0,\frac14)\cup(\frac14,\frac12)$, then,
			\begin{itemize}
				\item[(4-1).] the curvature equation \eqref{main problem} has no solutions if $t\not\in J(\frac1b)$;
				\item[(4-2).] the curvature equation \eqref{main problem} has a unique one-parameter scaling family of solutions $\{u_\beta(z)\}_{\beta>0}$ if $t\in J(\frac1b)$. 
			\end{itemize}
		\end{itemize}
		Furthermore, all the solutions obtained above satisfy $u_1(z)=u_1(-z)=u_1(\bar{z})$.
	\end{theorem}
	
	Together with Theorem \ref{theorem d}, Theorem \ref{main result general} generalizes  Theorem \ref{theorem c} to equation \eqref{main problem} successfully.

The approach of studying the curvature equation by exploring two conditional stability sets $\sigma_j=\Delta_j^{-1}([-1,1])$ of the associated second-order linear ODE was first introduced by Chen and Lin \cite{ChenLin2020}, and was developed further in \cite{ChenFuLin2021,ChenFuLin2026}. In this paper, we also apply this approach to prove Theorems \ref{main result for 0,1/4} and \ref{main result general}.
For the curvature equation \eqref{alpha=1,p not half period} studied in \cite{ChenFuLin2026}, the associated Hill discriminant $\Delta_j(A)$ is holomorphic for $A\in \mathbb C$. Thanks to the holomorphy of $\Delta_j(A)$ and the symmetries of the rectangular torus, the conditional stability sets $\sigma_j=\Delta_j^{-1}([-1,1])$ are not difficult to study in \cite{ChenFuLin2026}. However, when we consider the curvature equation \eqref{main problem}, the situation changes completely: The associated Hill discriminant $\Delta_j(A)$ is no longer an entire function on $\mathbb C$ but has an essential singularity at $A=0$, which makes the analysis of $\sigma_j=\Delta_j^{-1}([-1,1])$ much more difficult. To the best of our knowledge, such a situation has not been encountered in the literature of studying the curvature equations on torus. Therefore, new ideas are required to study the property of $\sigma_j=\Delta_j^{-1}([-1,1])$. Thanks to the symmetries of the rectangular torus, we can still determine $\sigma_2$ explicitly; see Corollary \ref{six intervals}. However, we can not determine $\sigma_1$. Fortunately, a delicate and novel analysis allows us to determine the structure of $\sigma_1\cap\sigma_2$ directly, which is quite surprising. Remark that	this approach is quite different from that of proving Theorem \ref{theorem d} in  \cite{Kuo2026}, where Kuo applied the idea of pre-modular forms from \cite{LinWangJEP}. Remark that to get a precise criterion of exsitence and non-existence of solutions, it seems that the idea of pre-modular forms works well only when all singular sources belong to $E_{\tau}[2]$ up to a translation. Thus, it seems that the idea of pre-modular forms can not be used to prove Theorems \ref{main result for 0,1/4} and \ref{main result general} for \eqref{main problem} with $2p\notin E_{\tau}[2]$ considered in this paper.

	The rest of this paper is organized as follow. In Section \ref{section 2}, we briefly review the basic theory of the generalized Lam\'e equation \eqref{generalized lame equation}, and define conditional stability sets $\sigma_j$. In Section \ref{section 3}, we discuss the integrability of the curvature equation \eqref{main problem} and establish a deep connection between \eqref{main problem} and the generalized Lam\'e equation \eqref{generalized lame equation}. The blow-up behavior of solutions to the curvature equation \eqref{main problem} is also given. In Section \ref{section 4}, we present some properties of the conditional stability sets $\sigma_j$ by proving that $A=0$ is an essential singularity of the Hill discriminant. In Section \ref{section 5}, we prove Theorem \ref{main result for 0,1/4} by analyzing $\sigma_1\cap\sigma_2$ via new ideas. This section is the most novel part of this paper.
Finally, Theorem \ref{main result general} will be proved in Section \ref{section 6}.

	\section{Preliminaries: generalized Lam\'e equation}\label{section 2}
	
	In this section, we briefly review some basic facts about the monodromy of the following generalized Lam\'e equation (GLE for short):
	\begin{equation}
		\label{generalized lame equation}
		\begin{aligned}
			y^{\prime\prime}(z) &= \left[ 2(\wp(z+p)+\wp(z-p)) + A(\zeta(z+p)-\zeta(z-p)) + B \right] y(z)\\
			&\eqqcolon I(z;p,A,\tau)y(z),
		\end{aligned}	
	\end{equation}
	where $p\not\in E_\tau[2]$ and, initially, $A,B\in\C$. 
	See e.g. \cite{ChaiLinWang2015, ChenLin2020} for similar arguments for the classical Lam\'{e} equation.
	See also \cite{Wang2026} for a theory of general second order generalized Lam\'{e} equations.
	
	Clearly GLE \eqref{generalized lame equation} is of Fuchsian type with regular singularities at $\pm p$. Since the local exponents of \eqref{generalized lame equation} at $\pm p$ are $-1$ and $2$, we need to figure out whether logarithmic singularities may occur,  and the apparent condition for GLE \eqref{generalized lame equation} is given as follows:
	
	\begin{lemma}\label{apparent condition lemma}
		The $\pm p$ are apparent singularities of GLE \eqref{generalized lame equation} (i.e., any solution of \eqref{generalized lame equation} has no logarithmic singularities at $\pm p$) if and only if
		\[
		A\left(B+\wp(2p)+A\zeta(2p)-\frac{A^2}{4}\right)+2\wp'(2p)=0.
		\]
	\end{lemma}
	\begin{proof}
Recall that the local exponents of GLE
\eqref{generalized lame equation} at $\pm p$ are $-1$ and $2$.
The exponent $2$ always gives a convergent Frobenius solution.
Thus, by Frobenius theory, $\pm p$ are apparent singularities if and
only if GLE \eqref{generalized lame equation} has a local solution
with exponent $-1$ at each of $\pm p$ of the form
\begin{equation}\label{frobenius sum}
    y(z)=\sum_{k=0}^{\infty}c_k(z\mp p)^{-1+k},\quad\text{with }c_0=1.
\end{equation}
We only need to establish this lemma for the point $p$.
Near $z=p$, it is well known that
\begin{equation*}
    \wp(z-p)=\frac{1}{(z-p)^2}
    +\frac{g_2}{20}(z-p)^2+O((z-p)^4),
\end{equation*}
\begin{equation*}
    \wp(z+p)=\wp(2p)+\wp^\prime(2p)(z-p)
    +\frac{1}{2}\wp^{\prime\prime}(2p)(z-p)^2+O((z-p)^3),
\end{equation*}
\begin{equation*}
    \zeta(z-p)=\frac{1}{z-p}
    -\frac{g_2}{60}(z-p)^3+O((z-p)^5)
\end{equation*}
and
\begin{equation*}
    \zeta(z+p)=\zeta(2p)-\wp(2p)(z-p)
    -\frac{1}{2}\wp^\prime(2p)(z-p)^2+O((z-p)^3).
\end{equation*}
Write the full Laurent expansion of the potential near $z=p$ as
\[
    I(z;p,A,\tau)=\frac{2}{(z-p)^2}-\frac{A}{z-p}
    +\sum_{j=0}^{\infty}I_j(z-p)^j,
\]
where $I_j$ are the Taylor coefficients of the holomorphic remainder.
In particular,
\[
    I_0=2\wp(2p)+A\zeta(2p)+B,
    \qquad I_1=2\wp^\prime(2p)-A\wp(2p).
\]
Substituting \eqref{frobenius sum} into GLE
\eqref{generalized lame equation} and comparing the coefficients of
$(z-p)^{k-3}$, we obtain
the recursive formula
\[
    k(k-3)c_k=-Ac_{k-1}
    +\sum_{j=0}^{k-2}I_jc_{k-2-j},
    \qquad k\geq1,
\]
where the sum is understood to be zero when $k=1$.
For $k=1,2$, this gives
\[
    c_1=\frac{A}{2},
    \qquad c_2=\frac{A^2}{4}-\frac{I_0}{2}.
\]
At the resonant index $k=3$, the coefficient of $c_3$ vanishes, so the
recurrence is compatible if and only if
\[
    0=-Ac_2+I_0c_1+I_1
    =-\frac{A^3}{4}+AI_0+I_1,
\]
or equivalently,
\begin{equation}\label{full apparent condition}
    A\left(B+\wp(2p)+A\zeta(2p)-\frac{A^2}{4}\right)
    +2\wp^\prime(2p)=0.
\end{equation}
When this condition holds, $c_3$ is arbitrary, and all $c_k$, $k\geq 4$, are determined by the recursive formula after $c_3$ is chosen.
Frobenius theory
ensures that the resulting series converges in a sufficiently small
punctured neighborhood of $p$. Together with the solution with exponent
$2$, it gives a local fundamental system without logarithmic singularities.
The condition at $-p$ is the same because the potential is even.
The proof is complete.
\end{proof}

\begin{remark} For the case $\wp'(2p)=0$ studied in \cite{Kuo2026}, the apparent condition \eqref{full apparent condition} is equivalent to either $A=0$ with $B\in\mathbb{C}$ arbitrary or 
\begin{equation}\label{2p=0B}
B=-\wp(2p)-A\zeta(2p)+\frac{A^2}{4},\quad A\in\mathbb{C}.
\end{equation} 
For the first case $A=0$ with $B\in\mathbb{C}$, the associated Hill discriminant $\Delta_j(B)$ of \eqref{generalized lame equation} is holomorphic as a function of $B\in\mathbb{C}$. For the second case \eqref{2p=0B}, the associated Hill discriminant $\Delta_j(A)$ of \eqref{generalized lame equation} is holomorphic as a function of $A\in\mathbb{C}$.
	
	Throughout the rest of this paper, we always assume that
	\begin{equation}\label{pcondition}
p\in E_{\tau}\setminus E_{\tau}[2],\qquad	\wp'(2p)\neq0,
	\end{equation}
	and the apparent condition \eqref{full apparent condition} holds, which is equivalent to
	\begin{equation}
		\label{apparent condition}
	A\in\mathbb{C}\setminus\{0\}\quad\text{and}\quad	B = \frac{A^2}{4} - A\zeta(2p) - \frac{2\wp^\prime(2p)}{A} - \wp(2p).
	\end{equation}
It turns out that the associated Hill discriminant $\Delta_j(A)$ of \eqref{generalized lame equation} is only holomorphic as a function of $A\in\mathbb{C}\setminus\{0\}$, and has an essential singularity at $A=0$; see Theorem \ref{essential singularity} for the proof. The presence of the essential singularity makes the problem more difficult to study, and requires us to develop different ideas; see Section \ref{section 5}.
	\end{remark}

\subsection{Monodromy matrices}	
By Lemma \ref{apparent condition lemma} and \eqref{pcondition}-\eqref{apparent condition}, we know the local monodromy matrix at $\pm p$ must be the identity matrix $I_2$. Thus, any solution $y(z)$ of GLE \eqref{generalized lame equation} has no branch points, which implies that $y(z)$ can be viewed as a single-valued meromorphic function in $\mathbb C$ by analytic continuation. Therefore, the monodromy representation of  \eqref{generalized lame equation} is a group homomorphism \begin{equation}\label{group homomorphism}
		\rho:\pi_1(E_\tau)\to SL(2,\C).
	\end{equation}
Take any linearly independent solutions $\tilde{y}_1(z),\tilde{y}_2(z)$ of \eqref{generalized lame equation}, there are two monodromy matrices $N_k=N_k(A)\in SL(2,\C)$ such that
\begin{equation}\nonumber
		\begin{pmatrix} \tilde{y}_1(z+1) \\ \tilde{y}_2(z+1) \end{pmatrix}
		= N_1 \begin{pmatrix} \tilde{y}_1(z) \\ \tilde{y}_2(z) \end{pmatrix},\quad\forall z\in\mathbb{C},
	\end{equation}
	\begin{equation}\nonumber
		\begin{pmatrix} \tilde{y}_1(z+\tau) \\ \tilde{y}_2(z+\tau) \end{pmatrix}
		= N_2 \begin{pmatrix} \tilde{y}_1(z) \\ \tilde{y}_2(z) \end{pmatrix},\quad\forall z\in\mathbb{C}.
	\end{equation}
Then the monodromy group is generated by $N_1$ and $N_2$. Clearly 
$$N_1 N_2 = N_2 N_1,$$
so there exists a common eigenfunction $y_1(z)$ of $N_1$ and $N_2$, i.e.,
	\begin{equation}\label{y1 monodromy}
		y_1(z+1)=\varepsilon_1y_1(z),\quad y_1(z+\tau)=\varepsilon_2y_1(z).
	\end{equation}
Clearly $y_2(z)\coloneqq y_1(-z)$ is also a solution of GLE \eqref{generalized lame equation} and
satisfies	\begin{equation}\label{y2 monodromy}
			y_2(z+1)=\varepsilon_1^{-1}y_2(z),\quad y_2(z+\tau)=\varepsilon_2^{-1}y_2(z),
	\end{equation}
namely $y_2(z)$ is also an eigenfunction of $N_k$ with eigenvalue $\varepsilon_k^{-1}$. Clearly, if $\varepsilon_k\neq \pm1$ for some $k$, then $y_1(z)$ and $y_2(z)$ are linearly independent. In general, $y_1(z)$ and $y_2(z)$ may be linearly dependent. Define
	\begin{equation}
		\Phi_e(z)\coloneqq y_1(z)y_2(z)=y_1(z)y_1(-z),
	\end{equation}
	then $\Phi_e(z)$ is an even elliptic function. A direct computation shows that $\Phi_e(z)$ satisfies the following second symmetric product equation for \eqref{generalized lame equation} (cf. \cite{ChaiLinWang2015}):
	\begin{equation}\label{second symmetric product equation}
		\Phi^{\prime\prime\prime}(z)-4I(z;p,A,\tau)\Phi^{\prime}(z)-2I^\prime(z;p,A,\tau)\Phi(z)=0.
	\end{equation}
	The following lemma provides more information on $\Phi_e(z)$.
	\begin{lemma}\label{dimension one}
		The dimension of the space of even elliptic solutions to \eqref{second symmetric product equation} is $1$. Moreover, up to multiplying a non-zero constant, we have
		\begin{align}\label{Phie exact expression}
			\Phi_e(z)=&\wp(z+p)+\wp(z-p)-A(\zeta(z+p)-\zeta(z-p))\nonumber\\
			&+\frac{A^2}{2}+A\zeta(2p)+\frac{2\wp^\prime(2p)}{A}-2\wp(2p).
		\end{align}
	\end{lemma}
	\begin{proof}
		Notice that we have already constructed an even elliptic solution $\Phi_e(z)$. For any nonzero even elliptic solution $\Phi(z)$ of \eqref{second symmetric product equation}, we know that $\Phi(z)$ can not be a constant, and the local exponents at $\pm p$ are the same and must be one of $\{-2,1,4\}$. Since a non-constant elliptic function must have poles, we see that $\Phi(z)$ must have poles of order $2$ both at $\pm p$. Define
		\begin{equation}\label{eqk1k2}
		k_1\coloneqq\lim_{z\to p}(z-p)^2\Phi_e(z)\neq 0,\quad k_2\coloneqq\lim_{z\to p}(z-p)^2\Phi(z)\neq 0.
		\end{equation}
	Then $k_2\Phi_e(z)-k_1\Phi(z)$ is also an even elliptic solution of \eqref{second symmetric product equation}. If $k_2\Phi_e(z)-k_1\Phi(z)\neq 0$, then $p$ is also a pole of $k_2\Phi_e(z)-k_1\Phi(z)$ with order $2$, a contradiction with \eqref{eqk1k2}. Thus $k_2\Phi_e(z)-k_1\Phi(z)=0$. This proves that the dimension of the space of even elliptic solutions to \eqref{second symmetric product equation} is $1$. 
		
		To prove \eqref{Phie exact expression},
up to multiplying a nonzero constant, we may assume $k_1=1$, i.e.,
		\begin{equation}\nonumber
			\Phi_e(z)=\frac{1}{(z-p)^2}+\frac{c_{-1}}{z-p}+c_0+c_1(z-p)+\cdots.
		\end{equation}		
		Recall the proof of Lemma \ref{apparent condition lemma} that
		\begin{equation}\nonumber
			I(z;p,A,\tau)=\frac{2}{(z-p)^2}-\frac{A}{z-p}+I_0+I_1(z-p)+I_2(z-p)^2+\cdots,
		\end{equation}
		\begin{equation}\nonumber
			I^\prime(z;p,A,\tau)=-\frac{4}{(z-p)^3}+\frac{A}{(z-p)^2}+I_1+2I_2(z-p)+\cdots.
		\end{equation}
Inserting these expansions into \eqref{second symmetric product equation} leads to $c_{-1}=A$, so we may assume that
		\begin{equation}\label{Phie formula original}
			\Phi_e(z)=\wp(z+p)+\wp(z-p)-A(\zeta(z+p)-\zeta(z-p))+d
		\end{equation}
		for some constant $d\in\C$. By inserting formula \eqref{Phie formula original} into \eqref{second symmetric product equation} and calculating the Laurent expansion at $z=p$, it is easy to deduce from the coefficient of $(z-p)^{-3}$ that
		\begin{equation}\nonumber
			d=\frac{A^2}{2}+A\zeta(2p)+\frac{2\wp^\prime(2p)}{A}-2\wp(2p).
		\end{equation}
		The proof is complete.
	\end{proof}
	\begin{lemma}\label{Q(A)}
		Let $\Phi_e(z)=y_1(z)y_2(z)=y_1(z)y_1(-z)$ be the even elliptic solution of \eqref{second symmetric product equation} given in \eqref{Phie exact expression}, and let $W=y_1(z)y_2^\prime(z)-y_1^\prime(z)y_2(z)$ be the Wronskian of $y_1$ and $y_2$. Then
		\begin{equation}\label{eqQA}
			Q(A)  \coloneqq\Phi_e^\prime(z)^2-2\Phi_e^{\prime\prime}(z)\Phi_e(z) + 4I(z; p, A, \tau)\Phi_e(z)^2
		\end{equation}
		 satisfies $W^2=Q(A)$
		and
		\begin{equation}\label{QA}
			Q(A)=\frac{(A^3-12\wp(2p)A+8\wp'(2p))
			(A^6-12\wp(2p)A^4+(48\wp(2p)^2-4g_2)A^2-16\wp'(2p)^2)}{4A^3}.
		\end{equation}
	\end{lemma}
	\begin{proof}
		By the symmetric product equation \eqref{second symmetric product equation}, we know $Q(A)$ is independent of $z$.
		To find the explicit expression of $Q(A)$, we set
		\[
		H(z)=\zeta(z+p)-\zeta(z-p)-\zeta(2p).
		\]
		Then the addition formulas of elliptic functions give 
		\[
		H^2=\wp(z+p)+\wp(z-p)+\wp(2p),\]\[
		H''=2H^3-6\wp(2p)H+2\wp'(2p),\]\[
		(H')^2=H^4-6\wp(2p)H^2+4\wp'(2p) H+g_2-3\wp(2p)^2.
		\]
	Consequently, $$I(z;p,A,\tau)=2H^2+AH+\frac{A^2}{4}-3\wp(2p)-\frac{2\wp'(2p)}{A},$$
		$$\Phi_e=H^2-AH+\frac{A^2}{2}-3\wp(2p)+\frac{2\wp'(2p)}{A}.$$ Inserting these formulas into \eqref{eqQA} yields \eqref{QA}.

		Since $\Phi_e(z)=y_1(z)y_2(z)$ and $W=y_1(z)y_2^\prime(z)-y_1^\prime(z)y_2(z)$, we have
		\begin{equation}\label{eqy1y2}
			\frac{y_1^\prime}{y_1}=\frac{\Phi_e^\prime-W}{2\Phi_e},\quad\quad	\frac{y_2^\prime}{y_2}=\frac{\Phi_e^\prime+W}{2\Phi_e}.
		\end{equation}
		Then
		\begin{equation}\nonumber
			\frac{\Phi_e^{\prime \prime}}{2 \Phi_e}-\frac{\Phi_e^{\prime}-W}{2 \Phi_e^2} \Phi_e^{\prime}=\left(\frac{y_1^{\prime}}{y_1}\right)^{\prime}=\frac{y_1^{\prime \prime}}{y_1}-\left(\frac{y_1^{\prime}}{y_1}\right)^2=I(z ; p, A, \tau)-\left(\frac{\Phi_e^{\prime}-W}{2 \Phi_e}\right)^2
		\end{equation}
		and
		\begin{equation}\nonumber
			\frac{\Phi_e^{\prime \prime}}{2 \Phi_e}-\frac{\Phi_e^{\prime}+W}{2 \Phi_e^2} \Phi_e^{\prime}=I(z ; p, A, \tau)-\left(\frac{\Phi_e^{\prime}+W}{2 \Phi_e}\right)^2.
		\end{equation}
		Adding the above two formulas, we obtain $W^2=Q(A)$.
		\end{proof}
	
	Let $\sigma(z)=\sigma(z ; \tau)$ be the Weierstrass sigma function, normalized by $\sigma^\prime/\sigma=\zeta$ and $\sigma^\prime(0)=1$, which is an odd entire function with simple zeros at $\Lambda_\tau$, and satisfies the following transformation law
	\begin{equation}\label{transformation law}
		\sigma(z+\omega_j)=-e^{\eta_j(z+\frac{1}{2} \omega_j)} \sigma(z), \quad j=1,2 .
	\end{equation}
	Then the explicit expression of the common eigenfunction $y_1(z)$ can be written down by $\sigma(z)$. 
	
	\begin{lemma}\label{information}
		Let $y_1(z)$ be the common eigenfunction mentioned above. Then there exist $a_1,a_2\in E_\tau\setminus\{\pm p\}$ satisfying $a_1\neq a_2$ in $E_\tau$ and 
		\begin{equation}\label{compatibility condition}
			2\zeta(a_1-a_2)=\zeta(a_1+p)+\zeta(a_1-p)-\zeta(a_2+p)-\zeta(a_2-p),
		\end{equation}
		\begin{equation}\label{condition on A}
			A = \zeta(a_1+p)-\zeta(a_1-p)+\zeta(a_2+p)-\zeta(a_2-p)-2\zeta(2p),
		\end{equation}
		such that up to multiplying a non-zero constant,
		\begin{equation}\label{y1 exact}
			y_1(z)=y_{a_1,a_2}(z)\coloneqq e^{c(a_1,a_2)z}\frac{\sigma(z-a_1)\sigma(z-a_2)}{\sigma(z-p)\sigma(z+p)},
		\end{equation}
		with
		\begin{equation}\label{condition on c}
			c(a_1,a_2)= \frac{1}{2}[\zeta(a_1+p) + \zeta(a_1-p) + \zeta(a_2+p) + \zeta(a_2-p)].
		\end{equation}
Consequently, $y_2(z)=y_1(-z)=y_{-a_1,-a_2}(z)$.
	\end{lemma}
	\begin{proof}
By \eqref{y1 monodromy}, we know that $y_1(z)$ is an elliptic function of the second kind. Its logarithmic derivative is elliptic, so the number of zeros equals the number of poles, counting multiplicities. Since $\Phi_e(z)=y_1(z)y_1(-z)$ has poles of order $2$  at both $p$ and $-p$, then $y_1$ has simple poles at both $p$ and $-p$. Every ordinary zero is simple by uniqueness for the initial-value problem, so $y_1(z)$ has exactly two distinct zeros, neither of which is $\pm p$. Then a classic theorem of elliptic function says that up to a constant, $y_1(z)$ can be written as
	$$y_1(z)=e^{cz}\frac{\sigma(z-a_1)\sigma(z-a_2)}{\sigma(z-p)\sigma(z+p)},$$
		with some $c\in\mathbb C$ and $a_1,a_2\in E_\tau\setminus\{\pm p\}$ satisfying that $a_1\neq a_2$ in $E_\tau$.
		
		Denote
		\begin{equation}\nonumber
			f(z)=\zeta(z-a_1)+\zeta(z-a_2)-\zeta(z-p)-\zeta(z+p),
		\end{equation}
		then we have
		\begin{equation}\nonumber
			\frac{y_1^\prime}{y_1}=c+f(z).
		\end{equation}
		Near $z=p$, we know
		\begin{equation}\nonumber
			f(z)=-\frac{1}{z-p}+\zeta(p-a_1)+\zeta(p-a_2)-\zeta(2p)+O(z-p).
		\end{equation}
		Since
		\begin{equation}\nonumber
			\left(\frac{y_1^{\prime}}{y_1}\right)^{\prime}+\left(\frac{y_1^{\prime}}{y_1}\right)^2=\frac{y_1^{\prime \prime}}{y_1}=I(z ; p, A, \tau),
		\end{equation}
		by comparing the coefficients of $(z-p)^{-1}$, we have
		\begin{equation}\nonumber
			2(c+\zeta(p-a_1)+\zeta(p-a_2)-\zeta(2p))=A.
		\end{equation}
		Near $z=-p$, by using a similar argument, we obtain that
		\begin{equation}\nonumber
			2(c+\zeta(-p-a_1)+\zeta(-p-a_2)-\zeta(-2p))=-A.
		\end{equation}
		Combining these two formulas and using $\zeta(-z)=-\zeta(z)$, we get \eqref{condition on A} and \eqref{condition on c}.
		
		By analyzing the behavior near $z=a_1$ and $z=a_2$ similarly, we can obtain the compatibility condition \eqref{compatibility condition} on $a_1$ and $a_2$.
		The proof is complete.
	\end{proof}
	Define $(r,s)=(r(A),s(A))\in\C^2$ by
	\begin{equation}\label{rs definition}
		-2\pi is\coloneqq c(a_1,a_2)-\eta_1(a_1+a_2),\quad 2\pi ir\coloneqq  c(a_1,a_2)\tau-\eta_2(a_1+a_2),
	\end{equation}
	which, by the Legendre relation $\tau\eta_1-\eta_2=2\pi i$, is equivalent to
	\begin{equation}\nonumber
		r+s\tau=a_1+a_2,\quad r\eta_1+s\eta_2=c(a_1,a_2).
	\end{equation}
	By applying the transformation law \eqref{transformation law} to \eqref{y1 exact}, we have
	\begin{equation}\nonumber
		y_{a_1,a_2}(z+1)=e^{c(a_1,a_2)-\eta_1a_1-\eta_1a_2}y_{a_1,a_2}(z)=e^{-2\pi is}y_{a_1,a_2}(z)
	\end{equation}
	and
	\begin{equation}\nonumber
		y_{a_1,a_2}(z+\tau)=e^{c(a_1,a_2)\tau-\eta_2a_1-\eta_2a_2}y_{a_1,a_2}(z)=e^{2\pi ir}y_{a_1,a_2}(z).
	\end{equation}
	Consequently, it follows from $y_{-a_1,-a_2}(z)=y_{a_1,a_2}(-z)$ that
	\begin{equation}\nonumber
		y_{-a_1,-a_2}(z+1)=e^{2\pi is}y_{-a_1,-a_2}(z)\quad\text{and}\quad 	y_{-a_1,-a_2}(z+\tau)=e^{-2\pi ir}y_{-a_1,-a_2}(z).
	\end{equation}
	There are two cases.
	
	\textbf{Case 1.} $Q(A)\neq0$. Then by Lemma \ref{Q(A)}, we know $y_{a_1,a_2}(z)$ and $y_{-a_1,-a_2}(z)$ are linearly independent, so it follows from \eqref{y1 exact} that $\{a_1,a_2\}\cap\{-a_1,-a_2\}=\emptyset$ (If $y_{a_1,a_2}(z)$ and $y_{-a_1,-a_2}(z)$ have a common zero, then the Wronskian $W=0$, a contradiction with $W^2=Q(A)\neq 0$). In particular, $a_j\notin E_{\tau}[2]$.  Furthermore, the monodromy matrices are given by
	\begin{equation}\nonumber
		N_1=N_1(A)=\begin{pmatrix} e^{-2\pi is}&0 \\ 0&e^{2\pi is}\end{pmatrix},\quad N_2=N_2(A)=\begin{pmatrix} e^{2\pi ir}&0 \\ 0&e^{-2\pi ir}\end{pmatrix}.
	\end{equation}
	We claim that $(r,s)\not\in\frac{1}2\mathbb{Z}^2=\{(r,s) : 2r, 2s\in\mathbb Z\}$. Suppose for contrary that $2r, 2s\in\mathbb{Z}$, then
	$N_k=\pm I_2$, so
	\begin{equation}\nonumber
		\Psi(z)\coloneqq y_{a_1,a_2}^2(z)+y_{-a_1,-a_2}^2(z)=y_{a_1,a_2}^2(z)+y_{a_1,a_2}^{2}(-z)
	\end{equation}
 is also an even elliptic solution of \eqref{second symmetric product equation}. By Lemma \ref{dimension one}, we know there exists a non-zero constant $c\in \C\setminus\{0\}$ such that $\Psi(z)=c\Phi_e(z)=cy_{a_1,a_2}(z)y_{-a_1,-a_2}(z)$, a contradiction with the fact that $y_{a_1,a_2}^2(z)$, $y_{-a_1,-a_2}^2(z)$ and $y_{a_1,a_2}(z)y_{-a_1,-a_2}(z)$ are linearly independent solutions of \eqref{second symmetric product equation}.
	
Therefore, in this case, we have $(r,s)\not\in\frac{1}{2}\mathbb{Z}^2$. In particular,
	\begin{equation}\nonumber
		(\operatorname{tr}N_1,\operatorname{tr}N_2)=(2\cos 2\pi s,2\cos 2\pi r)\not\in\{\pm(2,2),\pm(2,-2)\}.
	\end{equation}
	Clearly, $N_1(A)$ and $N_2(A)$ are both unitary matrices if and only if the corresponding monodromy data $(r,s)\in\R^2\setminus\frac{1}{2}\mathbb{Z}^2$.
	
	\textbf{Case 2.} $Q(A)=0$. Then $y_{a_1,a_2}(z)$ and $y_{-a_1,-a_2}(z)$ are linearly dependent, which, together with \eqref{y1 exact} and \eqref{y1 monodromy}-\eqref{y2 monodromy}, implies that $\{a_1,a_2\}=\{-a_1,-a_2\}$ in $E_{\tau}$ and $\varepsilon_1=e^{-2\pi is }\in\{\pm 1\}$,  $\varepsilon_2=e^{2\pi ir }\in\{\pm 1\}$, i.e., $(r,s)\in\frac{1}{2}\mathbb{Z}^2$.  We first determine the possible value of $A$. If $a_1=-a_1$ and $a_2=-a_2$ in $E_{\tau}$, there actually are $6$ possibilities:
	\begin{equation}\nonumber
		(a_1,a_2)=\left(\frac{\omega_j}{2},\frac{\omega_k}{2}\right),\quad 0\leq j<k\leq 3,
	\end{equation}
	with the corresponding
	\begin{equation}\nonumber
		A=A_{j,k}\coloneqq\zeta\left(\frac{\omega_j}{2}+p\right)-\zeta\left(\frac{\omega_j}{2}-p\right)+\zeta\left(\frac{\omega_k}{2}+p\right)-\zeta\left(\frac{\omega_k}{2}-p\right)-2\zeta(2p).
	\end{equation}
Here, we have used the fact that $a_1\neq a_2$. If $a_1=-a_2=a\in E_\tau\setminus E_\tau[2]$, by the compatibility condition \eqref{compatibility condition}, we have
	\begin{equation}\nonumber
		\zeta(2a)=\zeta(a+p)+\zeta(a-p).
	\end{equation}
	Define
	\begin{equation}\nonumber
		h(z)\coloneqq \zeta(2z)-\zeta(z+p)-\zeta(z-p).
	\end{equation}
	Since $\zeta(z+\omega_k)-\zeta(z)=\eta_k$, $k=1,2$, it is easy to see that $h(z)$ is an elliptic function with $6$ simple poles $\{\frac{\omega_l}{2}\}_{l=0}^3\cup\{\pm p\}$ in $E_\tau$. Notice that $h(z)$ is odd, so its six zeros, counted with multiplicity, can be denoted by $\{\pm b_1,\pm b_2,\pm b_3\}\subset E_\tau\setminus E_\tau[2]$. These pairs need not be distinct. Therefore, there are at most $3$ possibilities, listed with multiplicity:
	\begin{equation}\nonumber
		(a_1,a_2)=(b_k,-b_k)
	\end{equation}
	with the corresponding
	\begin{equation}\nonumber
		A=A_k\coloneqq 2[\zeta(b_k+p)-\zeta(b_k-p)-\zeta(2p)],\quad k=1,2,3.
	\end{equation}
	The above argument also shows
	\begin{equation}\label{eqzeros}
	Q(A)=0\quad\Longrightarrow \quad A\in\{A_{j,k} : 0\leq j<k\leq 3\}\cup\{A_k : k=1,2,3\}.
	\end{equation}
	
Notice that 
there can not be another common eigenfunction linearly independent of $y_{a_1,a_2}$. Otherwise both monodromy matrices $N_k$, whose eigenvalues are $\varepsilon_k\in\{\pm1\}$, would be $\pm I_2$. Let $v_e,v_o$ be the even and odd solutions normalized by $v_e(0)=v_o'(0)=1$ and $v_e'(0)=v_o(0)=0$. Then $v_e^2,v_o^2$ would be two linearly independent even elliptic solutions of \eqref{second symmetric product equation}, contradicting Lemma \ref{dimension one}. Therefore, up to a common conjugation,
	\begin{equation}\label{monodromy matrix 2}
		N_1(A)=\varepsilon_1\begin{pmatrix} 1&0 \\ 1&1\end{pmatrix},\quad N_2(A)=\varepsilon_2\begin{pmatrix} 1&0 \\ \mathcal{C}&1\end{pmatrix},
	\end{equation}
	where $\mathcal{C}\in \C\cup\{\infty\}$. Here, if $\mathcal{C}=\infty$, then \eqref{monodromy matrix 2} should be understood as
	\begin{equation}\nonumber
		N_1(A)=\varepsilon_1\begin{pmatrix} 1&0 \\ 0&1\end{pmatrix},\quad N_2(A)=\varepsilon_2\begin{pmatrix} 1&0 \\ 1&1\end{pmatrix}.
	\end{equation}
	In particular, 
	\begin{equation}\nonumber
		(\operatorname{tr}N_1,\operatorname{tr}N_2)=(2\varepsilon_1,2\varepsilon_2)\in\{\pm(2,2),\pm(2,-2)\}.
	\end{equation}
	Therefore, in this case, at least one of $N_1$ and $N_2$ has a nontrivial Jordan block and cannot be conjugated to a unitary matrix. 
	
\subsection{Conditional stability sets}	
	Define the Hill discriminants	
\begin{equation}\label{hilldis}
		\Delta_k(A)\coloneqq\frac12\operatorname{tr}N_k(A)=\begin{cases}\cos 2\pi s(A)\text{ if }k=1,\\ \cos2\pi r(A)\text{ if }k=2,\end{cases}\quad\forall A\in\mathbb{C}\setminus\{0\}.
	\end{equation}
Remark that $\Delta_k(A)$ is independent of the choice
of linearly independent solutions. We will prove in Theorem \ref{essential singularity} that $\Delta_k(A)$ has an essential singularity at $A=0$. Define
	\begin{equation}
		\sigma_k\coloneqq\Delta_k^{-1}([-1,1])=\{A\in\C\setminus\{0\}:\Delta_k(A)\in[-1,1]\},\quad k=1,2.\label{conditional stability set def}
	\end{equation}
Observe that
\begin{itemize}
\item When $A\in \sigma_1$, the corresponding $s(A)\in\mathbb{R}$, so $y_{1}(z)$ is bounded for $z\in q_0+\mathbb{R}$ as long as $q_0$ is chosen such that $q_0+\mathbb{R} \cap(\pm p+\Lambda_\tau)=\emptyset$. 
\item When $A\in \sigma_2$, the corresponding $r(A)\in\mathbb{R}$, so $y_{1}(z)$ is bounded for $z\in q_0+\tau\mathbb{R}$ as long as $q_0$ is chosen such that $q_0+\tau\mathbb{R} \cap(\pm p+\Lambda_\tau)=\emptyset$. 
\end{itemize}
	Thus, $\sigma_j$ can be considered as {\it the conditional stability sets} of GLE \eqref{generalized lame equation}.
Clearly the above arguments show that 
	\begin{equation} \label{eq: data5}
  \parbox{\dimexpr\linewidth-5em}{
 the monodromy of GLE \eqref{generalized lame equation} is unitary (i.e., the generators $N_1(A), N_2(A)\in SU(2)$ up to a common conjugation), if and only if $(r(A), s(A))\in \mathbb R^2\setminus\frac12\mathbb{Z}^2$, if and only if
$$
A\in (\sigma_1\cap\sigma_2)\setminus\{A:Q(A)=0\}.
$$
  }
\end{equation}

	\section{On the integrability theory of the curvature equation}\label{section 3}
	
	In this section, we establish the deep connection between GLE \eqref{generalized lame equation} and the following curvature equation:
	\begin{equation}\label{curvature equation}
		\Delta u+e^u=8\pi(\delta_p+\delta_{-p})\quad\text{on }E_\tau.
	\end{equation}
	The arguments of this section are quite standard, and are similar to those in \cite{ChaiLinWang2015, ChenLin2020, ChenFuLin2025, Kuo2026}. Here, we would like to give self-contained proofs for the reader's convenience. 

	\begin{theorem}\label{solution = unitary}
		Let $p\in E_\tau\setminus E_\tau[2]$ satisfy $\wp'(2p)\neq0$. Then the curvature equation \eqref{curvature equation} has solutions if and only if there exists $A\in\C\setminus\{0\}$ (and hence $B$ via \eqref{apparent condition}) such that the monodromy of GLE \eqref{generalized lame equation} is unitary. More precisely,
		\begin{itemize}
		\item[(1)] If $u(z)$ is a solution of \eqref{curvature equation}, then it belongs to a unique one-parameter scaling family of solutions $u_\beta(z)$, where $\beta>0$ is arbitrary, and $u_\beta(z)=u_\beta(-z)$ if and only if $\beta=1$. 
		\item[(2)] There is a one-to-one correspondence between those $A\in\C\setminus\{0\}$ such that the monodromy of GLE \eqref{generalized lame equation} is unitary and one-parameter scaling families of solutions (or equivalently, even solutions) of the curvature equation \eqref{curvature equation}.
		\end{itemize}
	\end{theorem}
	\begin{proof} We divide the proof into two steps.
	
	{\bf Step 1.} Let $A\in\C\setminus\{0\}$ such that the monodromy of GLE \eqref{generalized lame equation} is unitary. We prove the existence of a one-parameter scaling family of solutions $u_\beta(z)$, where $\beta>0$ is arbitrary.
	
Firstly, note that the local monodromy matrices at $\pm p$ can not be unitary matrices if solutions have logarithmic singularities, so there are no logarithmic singularities, namely	
$B$ is given by \eqref{apparent condition}. 

By the discussion in Section \ref{section 2}, there are linearly independent solutions $y_{a_1,a_2}(z)$, $y_{-a_1,-a_2}(z)$ of \eqref{generalized lame equation} and $(r,s)\in\R^2\setminus\frac{1}{2}\mathbb{Z}^2$ such that the monodromy matrices $N_1, N_2$ are given as follows:
$$\begin{pmatrix}
				y_{a_1,a_2}(z+1)\\y_{-a_1,-a_2}(z+1)
			\end{pmatrix}=\begin{pmatrix} e^{-2\pi is}&0 \\ 0&e^{2\pi is}\end{pmatrix}\begin{pmatrix}
				y_{a_1,a_2}(z)\\y_{-a_1,-a_2}(z)
			\end{pmatrix},$$
		\begin{equation}\label{monodromy in section 3}
			 \begin{pmatrix}
				y_{a_1,a_2}(z+\tau)\\y_{-a_1,-a_2}(z+\tau)
			\end{pmatrix}=\begin{pmatrix} e^{2\pi ir}&0 \\ 0&e^{-2\pi ir}\end{pmatrix}\begin{pmatrix}
				y_{a_1,a_2}(z)\\y_{-a_1,-a_2}(z)
			\end{pmatrix}.
		\end{equation}
		Define
		\begin{equation}\label{eqLiov}
			f(z)\coloneqq\frac{y_{a_1,a_2}(z)}{y_{-a_1,-a_2}(z)}=e^{2c(a_1,a_2)z}\frac{\sigma(z-a_1)\sigma(z-a_2)}{\sigma(z+a_1)\sigma(z+a_2)},
		\end{equation}
		where
		\begin{equation}\nonumber
			c(a_1,a_2)= \frac{1}{2}[\zeta(a_1+p) + \zeta(a_1-p) + \zeta(a_2+p) + \zeta(a_2-p)].
		\end{equation}
	Define $u_\beta(z)$ via the Liouville formula:
		\begin{equation}\label{u(z) expression}
			u_\beta(z)\coloneqq\log\frac{8|\beta f'(z)|^2}{(1+|\beta f(z)|^2)^2},\quad\forall \beta>0.
		\end{equation}
		We claim that $u_\beta(z)$ is a solution of \eqref{curvature equation}. By \eqref{monodromy in section 3}, we have
		\begin{equation}\nonumber
			f(z+1)=e^{-4\pi is}f(z),\quad	f(z+\tau)=e^{4\pi ir}f(z).
		\end{equation}
		Since $(r,s)\in\R^2$, we know $u(z)$ is doubly periodic and hence well-defined in $E_\tau$.
		Write
\begin{equation}\label{eqfpmp}f(z)=f(\pm p)+c_m(z\mp p)^{m}+O((z\mp p)^{m+1}),\end{equation} near $\pm p$ for some $m\geq 1$, $c_m\neq0$.
Since $f=y_{a_1,a_2}/y_{-a_1,-a_2}$ and $y_{\pm (a_1, a_2)}''=I(z;p,A,\tau)y_{\pm (a_1, a_2)}$, a direct computation gives
\begin{equation}\label{schwarz f}
			\{f;z\}\coloneqq\left(\frac{f''}{f'}\right)'-\frac{1}{2}\left(\frac{f''}{f'}\right)^2=-2I(z;p,A,\tau),
		\end{equation}
		where $\{f;z\}$ is called the Schwarzian derivative of $f(z)$.  Inserting \eqref{eqfpmp} into \eqref{schwarz f}, 
we get $m=3$, so it follows from \eqref{u(z) expression} that
		\begin{equation}\nonumber
			u_\beta(z)=4\log|z\mp p|+O(1)\quad\text{near }\pm p.
		\end{equation}
In particular, $f'(z)$ has exactly two double zeros $\pm p$.		
Notice from \eqref{eqLiov} that
\begin{equation}\frac{f'(z)}{f(z)}=\sum_{j=1}^2(\zeta(a_j+p) + \zeta(a_j-p)+\zeta(z-a_j)-\zeta(z+a_j))\end{equation}
is an elliptic function with totally four simple poles $\pm a_1, \pm a_2$ and two double zeros $\pm p$, so it has no other zeros, i.e., $f'(z)\neq 0$ for $z\in E_{\tau}\setminus\{\pm p\}$.
Then we see from \eqref{u(z) expression} that $u_\beta(z)\neq \infty$ for any $z\in E_{\tau}\setminus\{\pm p\}$, and a direct computation implies
\begin{align}
\Delta u_\beta+ e^{u_{\beta}}=0\quad\text{in}\quad E_{\tau}\setminus\{\pm p\}.
\end{align}
This proves that $u_\beta(z)$ is a one-parameter scaling family of solutions of \eqref{curvature equation}. Furthermore, since $f(-z)=\frac{1}{f(z)}$ but $\beta f(-z)\neq\frac{1}{\beta f(z)}$ for $\beta\neq 1$, a direct computation shows that $u_{\beta}(z)=u_\beta(-z)$ if and only if $\beta=1$.

{\bf Step 2.} Let $u(z)$ be a solution of \eqref{curvature equation}. We prove the existence of $A\in\C\setminus\{0\}$ such that the monodromy of GLE \eqref{generalized lame equation} is unitary.
Then by Step 1, there is a one-parameter scaling family of solutions $u_\beta(z)$ corresponding to this $A$. We prove the existence of $\beta_0>0$ such that $u(z)=u_{\beta_0}(z)$.

By the classical Liouville theorem, there is a locally meromorphic function $\tilde f(z)$ away from $\{\pm p\}$ such that
		\begin{equation}\label{liouville}
			u(z)=\log\frac{8|\tilde f'(z)|^2}{(1+|\tilde f(z)|^2)^2}.
		\end{equation}
		The map $\tilde f(z)$ is known as the developing map. By differentiating \eqref{liouville} directly, we have
		\begin{equation}\nonumber
			u_{zz}-\frac{1}{2}u_z^2=\{\tilde f;z\}=\left(\frac{\tilde f''}{\tilde f'}\right)'-\frac{1}{2}\left(\frac{\tilde f''}{\tilde f'}\right)^2.
		\end{equation}
	Notice that outside the singularities $\{\pm p\}$, it follows from $4u_{z\bar{z}}=\Delta u=-e^u$ that
		\begin{equation}\nonumber
			\left(u_{zz}-\frac{1}{2}u_z^2\right)_{\bar{z}}=(u_{z\bar{z}})_z-u_zu_{z\bar{z}}=-\frac{1}{4}(e^u)_z+\frac{1}{4}e^uu_z=0.
		\end{equation}
	Furthermore, it is easy to see that $v(z)\coloneqq u(z)-4\log|z\mp p|$ solves $\Delta v+|z\mp p|^4e^v=0$ near $\pm p$. The local exponential integrability in \cite[Corollary 1 and Remark 2]{BrezisMerle1991} gives $e^v\in L^q_{\mathrm{loc}}$ for every $q>1$. Elliptic regularity and bootstrapping then imply that $v$ is smooth at $z=\pm p$, thus
		\begin{equation}\nonumber
			u(z)=4\log|z\mp p|+O(1)\quad\text{near }\pm p,
		\end{equation}
and then
		\begin{equation}\nonumber
			u_{zz}-\frac{1}{2}u_z^2=\frac{-4}{(z\mp p)^2}+O\left(\frac{1}{z\mp p}\right)\quad\text{near }\pm p.
		\end{equation}
		Therefore, $u_{zz}-\frac{1}{2}u_z^2$ is an elliptic function with poles of order $2$ at $\{\pm p\}$. Since the sum of residues of any elliptic function is zero, we know there exists a constant $A=A(u)$ such that
		\begin{equation}\label{residue sum}
			2A=\operatorname{Res}_{z=p}\left(u_{zz}-\frac{1}{2}u_z^2\right)=-\operatorname{Res}_{z=-p}\left(u_{zz}-\frac{1}{2}u_z^2\right).
		\end{equation}
		Then, by Liouville theorem,  there exists a constant $B=B(u)$ such that
		\begin{equation}\nonumber
			\begin{aligned}
				u_{zz}-\frac{1}{2}u_z^2&
				=-2\left[ 2(\wp(z+p)+\wp(z-p)) + A(\zeta(z+p)-\zeta(z-p)) + B \right]\\
				&=-2I(z;p,A,\tau),
			\end{aligned}
		\end{equation}
		where $I(z;p,A,\tau)$ is given in \eqref{generalized lame equation}. We will explain below that $B$ is actually given by \eqref{apparent condition}.
		
		Since $\{\tilde f;z\}=-2I(z;p,A,\tau)$,  a classical result in complex analysis says that there are linearly independent solutions $\tilde y_1(z)$ and $\tilde y_2(z)$ of GLE \eqref{generalized lame equation} such that
		\begin{equation}\nonumber
			\tilde f(z)=\frac{\tilde y_1(z)}{\tilde y_2(z)}.
		\end{equation}
Define the Wronskian
		\begin{equation}\nonumber
			W\coloneqq \tilde y_1(z)\tilde y_2'(z)-\tilde y_1'(z)\tilde y_2(z),
		\end{equation}
		which is a non-zero constant. Inserting $\tilde f=\tilde y_1/\tilde y_2$ into the Liouville formula leads to
		\begin{equation}\nonumber
			2\sqrt{2}|W|e^{-\frac{1}{2}u(z)}=|\tilde y_1(z)|^2+|\tilde y_2(z)|^2.
		\end{equation}
		Since $u(z)$ is single-valued and doubly periodic, we immediately obtain that any monodromy matrix $N\in SL(2,\mathbb C)$ with respect to $(\tilde y_1, \tilde y_2)$ satisfy
		\begin{equation}\nonumber
			(\overline{\tilde y_1(z)}, \overline{\tilde y_2(z)})\overline{N}^{T}N\begin{pmatrix}\tilde y_1(z)\\ \tilde y_2(z)\end{pmatrix}=|\tilde y_1(z)|^2+|\tilde y_2(z)|^2,
		\end{equation}
so $\overline{N}^{T}N=I_2$, or equivalently, $N\in SU(2)$. Thus the monodromy group is a subgroup of $SU(2)$, namely the monodromy is unitary, and there are no logarithmic singularities. Consequently, Lemma \ref{apparent condition lemma} and $\wp'(2p)\neq0$ imply $A\neq0$ and give $B$ by \eqref{apparent condition}.
		
For this $A$, let $u_\beta(z)$ be the one-parameter scaling family of solutions constructed in Step 1.  We need to prove the existence of $\beta_0>0$ such that $u(z)=u_{\beta}(z)$.		
		
		Indeed, there is an invertible matrix $P=\bigl(\begin{smallmatrix}a & b\\
c & d\end{smallmatrix}\bigr)$ such that 
$$\begin{pmatrix}
\tilde y_{1}(z)\\
\tilde y_{2}(z)
\end{pmatrix}=\begin{pmatrix}
a&b\\
c&d
\end{pmatrix}\begin{pmatrix}
y_{a_1,a_2}(z)\\
y_{-a_1, -a_2}(z)
\end{pmatrix},$$
so
\begin{equation}\label{fft}\tilde f(z)=\frac{ay_{a_1,a_2}(z)+by_{-a_1,-a_2}(z)}{cy_{a_1,a_2}(z)+dy_{-a_1,-a_2}(z)}=\frac{a{f}(z)+b}{c{f}(z)+d},\end{equation}
and
$$P\begin{pmatrix}
e^{-2\pi is}&\\
&e^{2\pi is}
\end{pmatrix}P^{-1}=N_1,\quad P\begin{pmatrix}
e^{2\pi ir}&\\
&e^{-2\pi ir}
\end{pmatrix}P^{-1}=N_2.$$
Since $(r,s)\in\mathbb{R}^2\setminus \frac{1}{2}\mathbb{Z}^2$, without loss of generality, we may assume $r\in \mathbb{R}\setminus\frac12\mathbb Z$, i.e. $e^{2\pi ir}\neq \pm 1$. By $N_2\in SU(2)$, i.e. $\overline{N_2}^TN_2=I_2$, we obtain from $P\bigl(\begin{smallmatrix}e^{2\pi ir} & \\
 & e^{-2\pi ir}\end{smallmatrix}\bigr)=N_2P$ that
$$\overline{P}^TP\begin{pmatrix}
e^{2\pi ir}&\\
&e^{-2\pi ir}
\end{pmatrix}=\begin{pmatrix}
e^{2\pi ir}&\\
&e^{-2\pi ir}
\end{pmatrix}\overline{P}^TP,$$
which implies that $\overline{P}^TP$ is a diagonal matrix, i.e., $a\bar{b}+c\bar{d}=0$, so there is $\alpha\neq 0$ such that $(a,c)=\alpha (-\bar{d}, \bar{b})$. Consequently,
$$|ad-bc|=|\alpha| (|b|^2+|d|^2),\quad |a|^2+|c|^2=|\alpha|^2(|b|^2+|d|^2).$$
Let $\beta_0=|\alpha|>0$. From here and by inserting \eqref{fft} into \eqref{liouville}, we obtain 
{\allowdisplaybreaks
\begin{align*}
u(z)=&\log\frac{8|ad-bc|^2|{f}'(z)|^2}{(|a{f}(z)+b|^2+|c{f}(z)+d|^2)^2}\\
=&\log\frac{8|ad-bc|^2|{f}'(z)|^2}{(|b|^2+|d|^2+(|a|^2+|c|^2)|{f}(z)|^2)^2}\\
=&\log\frac{8|\beta_0 {f}'(z)|^2}{(1+|\beta_0 {f}(z)|^2)^2}=u_{\beta_0}(z).
\end{align*}
}%
The proof is complete.\end{proof}

To study the bubbling phenomenon of $u_\beta(z)$, we 
define a multiple Green function
	\begin{equation}\label{multiple green function definition}
		G_{2,p}(z_1,z_2)\coloneqq G(z_1-z_2)-\sum_{j=1}^2(G(z_j-p)+G(z_j+p)),
	\end{equation}
where $(z_1,z_2)\in(E_\tau\setminus\{\pm p\})^2$ with $z_1\neq z_2$. 

	\begin{remark}
		Note that if $(a_1,a_2)$ is a critical point of the multiple Green function $G_{2,p}(z_1,z_2)$, then $(-a_1,-a_2)$ is also a critical point. Here, we call a critical point $(a_1,a_2)$ is trivial if $\{a_1,a_2\}=\{-a_1,-a_2\}$ in $E_\tau$ (resp. non-trivial if $\{a_1,a_2\}\neq\{-a_1,-a_2\}$ in $E_\tau$).
	\end{remark}

	\begin{lemma}\label{solution - family of solutions}
		Let $p\in E_\tau\setminus E_\tau[2]$ satisfy $\wp'(2p)\neq0$. Let $A\in\C\setminus\{0\}$ such that the monodromy of GLE \eqref{generalized lame equation} is unitary. Let $u_\beta(z)$ be the one-parameter scaling family of solutions associated to this $A$ as constructed in Theorem \ref{solution = unitary}. Then
		\begin{itemize}
		\item[(1)] $u_\beta(z)$ blows up at $a_1$ and $a_2$ as $\beta\to+\infty$, and $u_\beta(z)$ blows up at $-a_1$ and $-a_2$ as $\beta\to0$.
		\item[(2)] $\pm(a_1, a_2)$ is a pair of non-trivial critical points of $G_{2,p}(z)$.
		\end{itemize}
	\end{lemma}
	
	\begin{proof}
	By \eqref{eqLiov}-\eqref{u(z) expression}, we immediately obtain
		\begin{equation}\nonumber
			e^{u_\beta(a_k)}=\frac{8|\beta{f}'(a_k)|^2}{(1+|\beta{f}(a_k)|^2)^2}\to+\infty\text{ as }\beta\to+\infty,\quad k=1,2,
		\end{equation}
		and $e^{u_\beta(z)}\to 0$ uniformly for $z\in K\Subset E_{\tau}\setminus\{a_1, a_2\}$ as $\beta\to+\infty$.
		Thus, $a_1$ and $a_2$ are both blow-up points of $u_\beta(z)$ as $\beta\to+\infty$. 
		Then it follows from
\cite{BartolucciTarantello2002,CL-1} that 
\[u_\beta(z)+8\pi(G(z-p)+G(z+p))-\frac{\int_{E_{\tau}}u_\beta}{|E_{\tau}|}=\int_{E_\tau}G(z-y)e^{u_\beta(y)}dy\to 8\pi\sum_{j=1}^2G(z-a_j)\]
uniformly in $K\Subset E_{\tau}\setminus\{a_1,a_2\}$.Then by using the same argument as \cite[Estimate B in pages 739-740]{CL-1} via
the
well-known Pohozaev identity,  we see that 
\[
\nabla G(a_j-p)+\nabla G(a_j+p)=\nabla G(a_{j}-a_{k}),\quad \{j,k\}=\{1,2\},\]
so the blow-up point set $(a_{1},a_{2})$ is a
nontrivial critical point of $G_{2,p}$.

Similarly, by letting $g=1/f$, we have $e^{u_\beta(z)}=\frac{8|\beta^{-1}{g}'(z)|^2}{(1+|\beta^{-1}{g}(z)|^2)^2}$, so
\begin{equation}\nonumber
			e^{u_\beta(-a_k)}\to+\infty\text{ as }\beta\to0,\quad k=1,2,
		\end{equation}
		and $e^{u_\beta(z)}\to 0$ uniformly for $z\in K\Subset E_{\tau}\setminus\{-a_1, -a_2\}$ as $\beta\to0$.
Thus, $-a_1$ and $-a_2$ are both blow-up points of $u_\beta(z)$ as $\beta\to0$. The proof is complete.
	\end{proof}

	\section{On the singularity near $A=0$}\label{section 4}

Recall the Hill discriminant $\Delta_k(A)$ defined for $A\in\mathbb{C}\setminus\{0\}$ in \eqref{hilldis}. The main purpose of this section is to prove that $A=0$ is an essential singularity of $\Delta_k(A)$. This is the main difference from \cite{ChenFuLin2026}. Recall that we always assume $\wp'(2p)\neq 0$.

	\begin{theorem}\label{essential singularity}
		$\Delta_k(A)$ is holomorphic for $A\in\mathbb C\setminus\{0\}$ but has an essential singularity at $A=0$.
	\end{theorem}
	\begin{proof} First, we consider $k=1$.
	Fix a base point $q_0\in E_\tau\setminus\{\pm p\}$ such that any of $q_0+\R$ and $q_0+\tau\R$  has no intersection with $\{\pm p\}+\Lambda_\tau$.
Then $I(x+q_0;p,A,\tau)$ is analytic for $x\in\R$. Define
	\begin{align*}
		\tilde{I}_1(x;A)\coloneqq& I(x+q_0;p,A,\tau)\\
		=&2(\wp(x+q_0+p)+\wp(x+q_0-p)) + A(\zeta(x+q_0+p)-\zeta(x+q_0-p)) \\
		&- \frac{2\wp^\prime(2p)}{A}+ \frac{A^2}{4} - A\zeta(2p)  - \wp(2p),
	\end{align*}
	and let $c(x)=c(x;A),s(x)=s(x;A)$ be the solutions of
	\begin{equation}\nonumber
		y^{\prime\prime}(x)=\tilde{I}_1(x;A) y(x)
	\end{equation}	
	satisfying the initial data
	\begin{equation}\nonumber
		c(0;A)=s^{\prime}(0;A)=1,\quad c^{\prime}(0;A)=s(0;A)=0.
	\end{equation}
	Then
	$$\begin{pmatrix}c(x+1;A)\\ s(x+1;A)\end{pmatrix}=\begin{pmatrix}c(1;A)&c'(1;A)\\ s(1;A)&s'(1;A)\end{pmatrix}\begin{pmatrix}c(x;A)\\ s(x;A)\end{pmatrix}.$$
Thus, $N_1(A)$ is conjugate to $\begin{pmatrix}c(1;A)&c'(1;A)\\ s(1;A)&s'(1;A)\end{pmatrix}$,	which implies that
	\begin{equation}\label{trace}
		\Delta_1(A)=\frac{1}{2}(c(1;A)+s'(1;A))
	\end{equation}
		and so $\Delta_1(A)$ is holomorphic for $A\in\mathbb C\setminus\{0\}$.
	
To prove that $A=0$ is an essential singularity of $\Delta_1(A)$, we set
			\begin{equation}\nonumber
		\mu=\sqrt{-\frac{2\wp'(2p)}{A}}\to\infty,\quad\text{as }\mathbb C\setminus\{0\}\ni A\to0.
	\end{equation}
		Denote
		\begin{equation}\nonumber
			\tilde I_1(x;A)=\mu^2+R(x,\mu),
		\end{equation}
		then there exist constants $M>0$ and $\mu_0>0$ such that for all $|\mu|\ge\mu_0$,
		\begin{equation}\label{potential}
			|R(x,\mu)|\le M,\quad\forall x\in[0,1].
		\end{equation}
		Now we consider the following perturbed equation
		\begin{equation}\nonumber
			y''(x)-\mu^2 y(x)=R(x,\mu) y(x),\quad x\in[0,1].
		\end{equation}
		By variation of constants, any solution $y(z)$ of the above equation can be rewritten as
		\begin{equation}\label{varconst}
			y(x)=y(0)\cosh\mu x+\frac{y'(0)}{\mu}\sinh\mu x
			+\frac{1}{\mu}\int_0^x \sinh\mu(x-t)\,R(t,\mu)\,y(t)\,dt.
		\end{equation}
		Applying this to $c(x)$ and $s(x)$, we obtain that
		\begin{align}
			c(x)&=\cosh\mu x+\frac{1}{\mu}\int_0^x \sinh\mu(x-t)\,R(t,\mu)\,c(t)\,dt,\label{cint}\\
			s(x)&=\frac{\sinh\mu x}{\mu}+\frac{1}{\mu}\int_0^x \sinh\mu(x-t)\,R(t,\mu)\,s(t)\,dt,\label{sint}\\ 
			s'(x)&=\cosh\mu x+\int_0^x \cosh\mu(x-t)\,R(t,\mu)\,s(t)\,dt.\label{sdint}
		\end{align}
		Let $\delta(x)\coloneqq c(x)-\cosh\mu x$. From \eqref{cint} we have
		\begin{equation}\nonumber
			\delta(x)=\frac{1}{\mu}\int_0^x \sinh\mu(x-t)\,R(t,\mu)(\cosh\mu t+\delta(t))\,dt.
		\end{equation}
		Using $|\sinh\mu(x-t)|\leq e^{|\operatorname{Re}\mu|(x-t)}$, $|\cosh\mu t|\le e^{|\operatorname{Re}\mu| t}$, and $|R(x,\mu)|\le M$, we get
	\begin{equation}\nonumber
		\begin{aligned}
			|\delta(x)|&\le\frac{M}{|\mu|}\int_0^x e^{|\operatorname{Re}\mu|(x-t)}(e^{|\operatorname{Re}\mu| t}+|\delta(t)|)\,dt\\
			&\le\frac{M}{|\mu|}x e^{|\operatorname{Re}\mu| x}+\frac{M}{|\mu|}\int_0^x e^{|\operatorname{Re}\mu|(x-t)}|\delta(t)|\,dt.
		\end{aligned}
	\end{equation}
	Define $\phi(x)\coloneqq e^{-|\operatorname{Re}\mu| x}|\delta(x)|$. Multiplying the inequality by $e^{-|\operatorname{Re}\mu| x}$ yields that
	\begin{equation}\nonumber
		\phi(x)\le\frac{M}{|\mu|}x+\frac{M}{|\mu|}\int_0^x\phi(t)dt\le\frac{M}{|\mu|}+\frac{M}{|\mu|}\int_0^x\phi(t)\,dt.
	\end{equation}
	By Gronwall's inequality, we know when $|\mu|$ is large, then
	\begin{equation}\nonumber
		\phi(x)\le\frac{M}{|\mu|}\exp\left(\frac{M}{|\mu|}x\right)\le\frac{2M}{|\mu|}.
	\end{equation}
	Consequently,
	\begin{equation}\nonumber
		|\delta(x)|\le\frac{2M}{|\mu|}e^{|\operatorname{Re}\mu| x},\quad x\in[0,1].
	\end{equation}
	In particular,
	\begin{equation}\nonumber
		|c(1)-\cosh\mu|\leq\frac{2M}{|\mu|}e^{|\operatorname{Re}\mu| }.
	\end{equation}
	
	Similarly, let $\delta(x)=s(x)-\frac{\sinh\mu x}{\mu}$. From (\ref{sint}) we similarly obtain that
		\begin{equation}\nonumber
			|\delta(x)|\le\frac{M}{|\mu|}\int_0^x e^{|\operatorname{Re}\mu|(x-t)}\left(\frac{e^{|\operatorname{Re}\mu| t}}{|\mu|}+|\delta(t)|\right)\,dt
			\le\frac{M}{|\mu|^2}e^{|\operatorname{Re}\mu| x}+\frac{M}{|\mu|}\int_0^x e^{|\operatorname{Re}\mu|(x-t)}|\delta(t)|\,dt.
		\end{equation}
		Gronwall's inequality gives for large $|\mu|$ that
		\begin{equation}\nonumber
			e^{-|\operatorname{Re}\mu| x}|\delta(x)|\le\frac{2M}{|\mu|^2}.
		\end{equation}
		Therefore, we know
	\begin{equation}\nonumber
		|s(x)|\le\frac{2}{|\mu|}e^{|\operatorname{Re}\mu| x},\quad x\in[0,1],
	\end{equation}
	from which and \eqref{sdint} we get
	\begin{equation}\nonumber
		\begin{aligned}
			|s'(x)-\cosh\mu x|&	\le M\int_0^x e^{|\operatorname{Re}\mu|(x-t)}|s(t)|\,dt\\
			&\le M\int_0^x e^{|\operatorname{Re}\mu|(x-t)}\frac{2}{|\mu|}e^{|\operatorname{Re}\mu| t}\,dt\\
			&\leq\frac{2M}{|\mu|}e^{|\operatorname{Re}\mu| x}.
		\end{aligned}
	\end{equation}
	In particular, at $x=1$, we have
	\begin{equation}\nonumber
		|s'(1)-\cosh\mu|\le\frac{2M}{|\mu|}e^{|\operatorname{Re}\mu| }.
	\end{equation}
	Combining the above discussion, we know
	\begin{equation}\label{asymptotic 1}
		\Delta_1(A)=\frac{1}{2}(c(1;A)+s'(1;A))=\cosh \mu+ E_1(\mu),
	\end{equation}
	where
	\begin{equation}\label{error term 1}
		|E_1(\mu)|\leq\frac{2M}{|\mu|}e^{|\operatorname{Re}\mu| }.
	\end{equation}	
		
In particular, if we take 	
\begin{equation}\label{param}
			A=-\frac{2\wp'(2p)}{\mu^2},\quad \mu>0,\quad\mu\to+\infty,
		\end{equation}		
	then
		\begin{equation}\nonumber
			\Delta_1(A)
			=\frac12 e^{\mu}\left(1+O\left(\frac1\mu\right)\right)\sim\frac12\exp\left(\sqrt{-\frac{2\wp'(2p)}{A}}\right),
			\quad \mu\to+\infty.
		\end{equation}
Thus, $A=0$ is an essential singularity of $\Delta_1(A)$.

The case $k=2$ can be proved similarly. Recall that $q_0$ is a fixed base point such that $q_0+\tau\R$ has no intersection with $\{\pm p\}+\Lambda_\tau$, so $I(\tau x+q_0;p,A,\tau)$ is analytic for $x\in\R$. Define
	\begin{equation}\nonumber
		\tilde{I}_2(x;A)\coloneqq \tau^2 I(\tau x+q_0;p,A,\tau)
	\end{equation}
	and let $c_2(x)=c_2(x;A),s_2(x)=s_2(x;A)$ be the solutions of
	\begin{equation}\nonumber
		y^{\prime\prime}(x)=\tilde{I}_2(x;A) y(x)
	\end{equation}	
	satisfying the initial data
	\begin{equation}\nonumber
		c_2(0;A)=s_2^{\prime}(0;A)=1,\quad c_2^{\prime}(0;A)=s_2(0;A)=0.
	\end{equation}
Similarly, we have
	\begin{equation}\nonumber
		\Delta_2(A)=\frac{1}{2}(c_2(1;A)+s_2'(1;A)).
	\end{equation}
Set
	\begin{equation}\nonumber
		\mu=\sqrt{-\frac{2\wp'(2p)}{A}}\to\infty,\quad\text{as }\C\setminus\{0\}\ni A\to0,
	\end{equation}
	and denote similarly
	\begin{equation}\nonumber
		\tilde I_2(x;A)=\tau^2\mu^2+R_2(x,\mu),
	\end{equation}
	then there exist constants $M>0$ and $\mu_0>0$ such that for all $|\mu|\ge\mu_0$,
	\begin{equation}
		|R_2(x,\mu)|\le M,\quad\forall x\in[0,1].
	\end{equation}
Then a similar argument as above implies
	\begin{equation}\label{asymptotic 2}
		\Delta_2(A)=\frac{1}{2}(c_2(1;A)+s_2'(1;A))=\cosh \tau\mu+ E_2(\mu),
	\end{equation}
	where
	\begin{equation}\label{error term 2}
		|E_2(\mu)|\leq\frac{2M}{|\tau\mu|}e^{|\operatorname{Re}\tau\mu| }.
	\end{equation}
Consequently, taking $\mathbb R\ni \tau\mu\to+\infty$ shows that $A=0$ is also an essential singularity of $\Delta_2(A)$.
	\end{proof}

Therefore, the behaviors of $\Delta_1$ and $\Delta_2$ near $A=0$ may be very complicated. Here, we can prove that the conditional stability sets $\sigma_1$ and $\sigma_2$ are separated near the origin. 

	\begin{theorem}\label{cap is empty near 0} There exists $\varepsilon>0$ such that 
		\begin{equation}\nonumber
			\sigma_1\cap\sigma_2\cap\{A\in\C:0<|A|<\varepsilon\}=\emptyset.
		\end{equation}
	\end{theorem}
	\begin{proof}
		Suppose for contrary that there exist $\{A_n\}_{n=1}^\infty\subset\C\setminus\{0\}$ satisfying that $A_n\to0$ and $A_n\in\sigma_1\cap\sigma_2$. Denote
		\begin{equation}\nonumber
			\mu_n=\sqrt{-\frac{2\wp'(2p)}{A_n}}\to\infty,\quad\text{as }n\to\infty.
		\end{equation}
		By \eqref{hilldis}, we know for any $n\in\mathbb{N}_+$,
		\begin{equation}\nonumber
			|\Delta_1(A_n)|\leq 1\quad\text{and}\quad |\Delta_2(A_n)|\leq 1.
		\end{equation}
		By \eqref{asymptotic 1} and \eqref{error term 1} and the triangle inequality, we have
		\begin{equation}\nonumber
			\begin{aligned}
				1\geq |\Delta_1(A_n)|&\geq|\cosh\mu_n|-\frac{2M}{|\mu_n|}e^{|\operatorname{Re}\mu_n |}\\
				&=\left|\frac{e^{\mu_n}+e^{-\mu_n}}{2}\right|-\frac{2M}{|\mu_n|}e^{|\operatorname{Re}\mu_n|}\\
				&\geq\frac{e^{|\operatorname{Re}\mu_n|}}{2}-\frac{e^{-|\operatorname{Re}\mu_n|}}{2}-\frac{2M}{|\mu_n|}e^{|\operatorname{Re}\mu_n|},
			\end{aligned}
		\end{equation}
		from which we know $|\operatorname{Re}\mu_n|$ is bounded for all $n\in\mathbb{N}_+$. This implies $|\operatorname{Im}\mu_n|\to +\infty$ and then
		$$|\operatorname{Re}\tau\mu_n|=|\operatorname{Im}\tau\operatorname{Im}\mu_n-\operatorname{Re}\tau\operatorname{Re}\mu_n|\to +\infty,$$
		because $\operatorname{Im}\tau>0$.
		Consequently, it follows from \eqref{asymptotic 2} and \eqref{error term 2} that
		\begin{equation}\nonumber
			\begin{aligned}
				1\geq |\Delta_2(A_n)|&\geq|\cosh\tau\mu_n|-\frac{2M}{|\tau\mu_n|}e^{|\operatorname{Re}\tau\mu_n|}\\
				&=\left|\frac{e^{\tau\mu_n}+e^{-\tau\mu_n}}{2}\right|-\frac{2M}{|\tau\mu_n|}e^{|\operatorname{Re}\tau\mu_n|}\\
				&\geq\frac{e^{|\operatorname{Re}\tau\mu_n|}}{2}-\frac{e^{-|\operatorname{Re}\tau\mu_n|}}{2}-\frac{2M}{|\tau\mu_n|}e^{|\operatorname{Re}\tau\mu_n|}\to+\infty,
			\end{aligned}
		\end{equation}
 a contradiction. The proof is complete.
	\end{proof}
	Theorem \ref{cap is empty near 0} tells us that, although $A=0$ is an essential singularity for $\Delta_1(A)$ and $\Delta_2(A)$, it is still possible to determine the structure of $\sigma_1\cap\sigma_2$.

	\section{Analysis on the conditional stability sets}\label{section 5}
	
	From now on, we always assume that $\tau=ib$ with $b>0$, that is, $E_\tau$ is a rectangular torus. Then it follows from the definition of $\wp(z)$ and $\zeta(z)$ that
	\begin{equation}\label{eqwpz}\overline{\wp(\bar{z})}=\wp(z),\quad\overline{\zeta(\bar{z})}=\zeta(z).\end{equation}
In particular, $\wp(z)\in\mathbb{R}$ if and only if $\wp(z)=\wp(\bar z)$, if and only if $z=\pm \bar z$ in $E_{\tau}$, which is equivalent to 
$$\pm z\in \Big(0,\frac{1}{2} \Big]\cup  \Big[\frac{1}{2},\frac{1+\tau}{2} \Big]\cup  \Big[\frac{1+\tau}{2},\frac{\tau}{2} \Big]\cup  \Big[\frac{\tau}{2}, 0\Big),\quad\operatorname{mod}\;\mathbb{Z}+\mathbb{Z}\tau.$$
Here
$[z_{1},z_{2}]\coloneqq\{(1-t)z_{1}+tz_{2} :\;0\leq t\leq1\}$, and $[z_1,z_2)$, $(z_1, z_2]$, $(z_1,z_2)$ are defined similarly. Furthermore, $\wp$ is one to
one from $(0,\frac{\tau}{2}]\cup \lbrack
\frac{\tau}{2},\frac{1+\tau}{2}]\cup \lbrack \frac{1+\tau}{2},\frac{1}{2}%
]\cup \lbrack \frac{1}{2},0)$ onto $(-\infty,+\infty)$, with
\[
\lim_{\,[\frac{1}{2},0)\ni z\rightarrow0}\wp(z)=+\infty,\quad \lim
_{(0,\frac{\tau}{2}]\ni z\rightarrow0}\wp(z)=-\infty.
\] 
Recall $e_k=\wp(\frac{\omega_k}{2})$ for $k=1,2,3$ and
\begin{align*}
 &\wp'(z)^2 = 4\wp(z)^3 - g_2\wp(z) - g_3=4\prod_{k=1}^3(\wp(z)-e_k), \end{align*}
 we have $e_{j}, g_2, g_3
\in \mathbb{R}$, $e_{2}<e_{3}<e_{1}$, $g_2>0$ and $e_{2}<0<e_{1}$. Moreover, $\wp^{\prime
}(z)\in \mathbb{R}$ if $z\in(0,\frac
{1}{2}]\cup[\frac{\tau}{2},\frac{1+\tau}{2}]$, and
$\wp^{\prime}(z)\in i\mathbb{R}$ if
$z\in(0,\frac{\tau}{2}]\cup[\frac{1}{2},\frac{1+\tau}{2}]$. Similarly,
$\zeta(z)\in \mathbb{R}$ for $z\in (0,1)$,
so $\eta_{1}\in \mathbb{R}$ and $\eta_2=\tau\eta_1-2\pi i\in i\mathbb{R}$. We will freely use these facts below. 

Besides, in the following proofs we will frequently use the following formulas for elliptic functions:
\begin{equation}\label{add1}\zeta(2z)-2\zeta(z)=\frac{\wp''(z)}{2\wp'(z)},\end{equation}
\begin{equation}\label{add2}\zeta(z+w)+\zeta(z-w)-2\zeta(z)=\frac{\wp'(z)}{\wp(z)-\wp(w)},\end{equation}
\begin{equation}\label{add3}\zeta(z+w)-\zeta(z)-\zeta(w)=\frac12\frac{\wp'(z)-\wp'(w)}{\wp(z)-\wp(w)},\end{equation}
\begin{equation}\label{add4}
			\sum_{k=1}^{3}\frac{\wp'(p)}{2(\wp(p)-e_k)}=\frac{\wp''(p)}{\wp'(p)},
		\end{equation}
\begin{equation}\label{add5}
		\wp(2z)=\frac{\wp''(z)^2-8\wp(z)\wp'(z)^2}{4\wp'(z)^2},
		\end{equation}
		\begin{equation}\label{add6}
		\wp'(2z)=-\frac{\wp''(z)^3-12\wp(z)\wp'(z)^2\wp''(z)+4\wp'(z)^4}{4\wp'(z)^3}.
		\end{equation}
	
In this section, we consider $p\in(0,1/4)$ and prove Theorem \ref{main result for 0,1/4}. By \eqref{eq: data5} and Theorem \ref{solution = unitary}, we need to count $$\#((\sigma_1\cap\sigma_2)\setminus\{A:Q(A)=0\}).$$
	As mentioned before, the previous approaches of analyzing $\sigma_1\cap\sigma_2$ in \cite{ ChenFuLin2021, ChenFuLin2026, ChenLin2020} can not work due to the presence of the essential singularity $A=0$. Here, we need to develop new ideas to overcome this difficulty. This section is the most novel part of this paper.

\subsection{Zeros of $Q(A)$}	We begin by investigating the zeros of $Q(A)$. 
	\begin{theorem}
		$Q(A)$ has exactly nine distinct real zeros $\{\beta_1,\beta_2,\beta_3,\pm\alpha_1,\pm\alpha_2,\pm\alpha_3\}$ satisfying that
		\begin{equation}\label{inequality}
			\beta_1<-\alpha_2<-\alpha_3<-\alpha_1<\beta_2<0<\alpha_1<\alpha_3<\alpha_2<\beta_3.
		\end{equation}
		Here,
		\begin{equation}\nonumber
			\alpha_k\coloneqq -\frac{\wp''(p)}{\wp'(p)}+\frac{\wp'(p)}{\wp(p)-e_k},\quad k=1,2,3,
		\end{equation}
		satisfies
		\begin{equation}\label{eqr1}
			R_1(A)\coloneqq A^6-12\wp(2p)A^4+(48\wp(2p)^2-4g_2)A^2-16\wp'(2p)^2=\prod_{j=1}^3(A^2-\alpha_j^2),
		\end{equation}
		and $\beta_1<\beta_2<\beta_3$ are three real zeros of the polynomial
		\begin{equation}\nonumber
			R_2(A)\coloneqq A^3-12\wp(2p)A+8\wp'(2p).
		\end{equation}
	\end{theorem}
	\begin{proof}
		Suppose that $Q(A)=0$, then \eqref{eqzeros} implies that either \begin{equation}\nonumber
			A=A_{j,k}\coloneqq\zeta\left(\frac{\omega_j}{2}+p\right)-\zeta\left(\frac{\omega_j}{2}-p\right)+\zeta\left(\frac{\omega_k}{2}+p\right)-\zeta\left(\frac{\omega_k}{2}-p\right)-2\zeta(2p)
		\end{equation}
		for $0\leq j<k\leq 3$, or 
		\begin{equation}\nonumber
			A=A_k\coloneqq 2[\zeta(b_k+p)-\zeta(b_k-p)-\zeta(2p)]\quad\text{for }k=1,2,3.
		\end{equation}

	{\bf Step 1.} We prove that $\{A_{j,k}\}_{j<k}=\{\pm \alpha_k\}_{k=1}^3$ and they are pairwise distinct.
		
Assume first that $j=0$ and $k\in\{1,2,3\}$. Then by $\zeta(-z)=-\zeta(z)$ and \eqref{add1}-\eqref{add2}, we have
		\begin{equation}\nonumber
			\begin{aligned}
				A_{0,k}&=2\zeta(p)+\zeta\left(p+\frac{\omega_k}{2}\right)+\zeta\left(p-\frac{\omega_k}{2}\right)-2\zeta(2p)\\
				&=2(2\zeta(p)-\zeta(2p))+\frac{\wp'(p)}{\wp(p)-e_k}\\
				&=-\frac{\wp''(p)}{\wp'(p)}+\frac{\wp'(p)}{\wp(p)-e_k}
				\eqqcolon \alpha_k.
			\end{aligned}
		\end{equation}
		Obviously, $\alpha_k\in\R$ since $p\in(0,1/4)$, $g_2\in\R$ and $e_k\in \R$. Also, since $e_2<e_3<e_1<\wp(p)$ and $\wp'(p)<0$, we have $\alpha_1<\alpha_3<\alpha_2$. 
		
		Now, suppose that $\{i,j,k\}=\{1,2,3\}$ and $i<j$, then by \eqref{add1}, \eqref{add2} and \eqref{add4}, 
		\begin{equation}\nonumber
			\begin{aligned}
				A_{i,j}&=\zeta\left(p+\frac{\omega_i}{2}\right)+\zeta\left(p-\frac{\omega_i}{2}\right)+\zeta\left(p+\frac{\omega_j}{2}\right)+\zeta\left(p-\frac{\omega_j}{2}\right)-2\zeta(2p)\\
				&=2(2\zeta(p)-\zeta(2p))+\frac{\wp'(p)}{\wp(p)-e_i}+\frac{\wp'(p)}{\wp(p)-e_j}\\
				&=-\frac{\wp''(p)}{\wp'(p)}+\frac{\wp'(p)}{\wp(p)-e_i}+\frac{\wp'(p)}{\wp(p)-e_j}\\
				&=\frac{\wp''(p)}{\wp'(p)}-\frac{\wp'(p)}{\wp(p)-e_k}
				=-\alpha_k.
			\end{aligned}
		\end{equation}
Therefore, $\{A_{j,k}\}_{j<k}=\{\pm \alpha_k\}_{k=1}^3$. 

Next, we prove that they are pairwise distinct. For simplicity, denote
		\begin{equation}\nonumber
			\alpha_k=-\frac{\wp''(p)}{\wp'(p)}+\frac{\wp'(p)}{\wp(p)-e_k}=C+t_k,\quad k=1,2,3.
		\end{equation}
		Then, \eqref{add4} gives $t_1+t_2+t_3=-2C$. Besides,
		\begin{equation}\nonumber
			\begin{aligned}
				\alpha_1\alpha_2\alpha_3&=(C+t_1)(C+t_2)(C+t_3)\\
				&=-C^3+(t_1t_2+t_1t_3+t_2t_3)C+t_1t_2t_3.
			\end{aligned}
		\end{equation}
		Since $e_1+e_2+e_3=0$, it is easy to see that
		\begin{equation}\nonumber
			\begin{aligned}
				t_1t_2+t_1t_3+t_2t_3&=\frac{\wp'(p)^2(\wp(p)-e_1+\wp(p)-e_2+\wp(p)-e_3)}{(\wp(p)-e_1)(\wp(p)-e_2)(\wp(p)-e_3)}\\
				&=\frac{3\wp(p)\wp'(p)^2}{\frac{1}{4}\wp'(p)^2}=12\wp(p),
			\end{aligned}
		\end{equation}
		and
		\begin{equation}\nonumber
			t_1t_2t_3=\frac{\wp'(p)^3}{(\wp(p)-e_1)(\wp(p)-e_2)(\wp(p)-e_3)}=4\wp'(p).
		\end{equation}
		Therefore, 
		\begin{equation}\nonumber
			\begin{aligned}
				\alpha_1\alpha_2\alpha_3&=\frac{\wp''(p)^3}{\wp'(p)^3}-12\wp(p)\frac{\wp''(p)}{\wp'(p)}+4\wp'(p)\\
				&=\frac{\wp''(p)^3+4\wp'(p)^4-12\wp(p)\wp'(p)^2\wp''(p)}{\wp'(p)^3}\\
				&=-4\wp'(2p)\neq0,
			\end{aligned}
		\end{equation}
	where we have used \eqref{add6} to obtain the last equality. Hence, $\alpha_k\neq0$ for $k=1,2,3$. Obviously, $\alpha_i\neq\alpha_j$ if $i\neq j$. On the other hand, for $\{i,j,k\}=\{1,2,3\}$,
		\begin{equation}\nonumber
			\alpha_i+\alpha_j=2C+t_i+t_j=-t_k=-\frac{\wp'(p)}{\wp(p)-e_k}\neq0,
		\end{equation}
		and thus $\alpha_i\neq -\alpha_j$ for $i\neq j$. 
		
		{\bf Step 2.} We prove that $R_1(A)=\prod_{j=1}^3(A^2-\alpha_j^2)$.
		
	By \eqref{add4}-\eqref{add5}, we have
		\begin{equation}\nonumber
			\begin{aligned}
				\alpha_1^2+\alpha_2^2+\alpha_3^2&=(\alpha_1+\alpha_2+\alpha_3)^2-2(\alpha_1\alpha_2+\alpha_1\alpha_3+\alpha_2\alpha_3)\\
				&=C^2-6C^2-4(t_1+t_2+t_3)C-2t_1t_2-2t_1t_3-2t_2t_3\\
				&=3C^2-24\wp(p)=12\wp(2p).
			\end{aligned}
		\end{equation}
		In particular, 
		$$\alpha_1\alpha_2+\alpha_1\alpha_3+\alpha_2\alpha_3=12\wp(p)-C^2.$$
Consequently,
\[
\begin{aligned}
    \sum_{i<j}\alpha_i^2\alpha_j^2
    &=\bigg(\sum_{i<j}\alpha_i\alpha_j\bigg)^2
      -2\alpha_1\alpha_2\alpha_3\sum_{i=1}^3\alpha_i\\
    &=(12\wp(p)-C^2)^2+8C\wp'(2p).
\end{aligned}
\]
Since \eqref{add5}-\eqref{add6} give
		\begin{equation}\nonumber
			\wp(2p)=-2\wp(p)+\frac{\wp''(p)^2}{4\wp'(p)^2}=-2\wp(p)+\frac{1}{4}C^2,
		\end{equation}
		and
		\begin{equation}\nonumber
			\begin{aligned}
				2C^3-24\wp(p)C-8\wp'(p)&=-2\frac{\wp''(p)^3}{\wp'(p)^3}+24\frac{\wp(p)\wp''(p)}{\wp'(p)}-8\wp'(p)\\
				&=\frac{-2\wp''(p)^3+24\wp(p)\wp'(p)^2\wp''(p)-8\wp'(p)^4}{\wp'(p)^3}\\
				&=8\wp'(2p),
			\end{aligned}
		\end{equation}
it follows from 
$C\wp'(p)=-\wp''(p)=-6\wp(p)^2+g_2/2$ that
\[
\begin{aligned}
    \sum_{i<j}\alpha_i^2\alpha_j^2
    &=(12\wp(p)-C^2)^2
      +C(2C^3-24\wp(p)C-8\wp'(p))\\
    &=3C^4-48\wp(p)C^2+144\wp(p)^2
      -8\left(-6\wp(p)^2+\frac{g_2}{2}\right)\\
    &=48\left(\frac{C^2}{4}-2\wp(p)\right)^2-4g_2=48\wp(2p)^2-4g_2.
\end{aligned}
\]	
Together all the above formulas, we get \eqref{eqr1}, i.e., $R_1(A)=\prod_{j=1}^3(A^2-\alpha_j^2)$.

{\bf Step 3.} We prove that $R_2(A)=A^3-12\wp(2p)A+8\wp'(2p)$ has three distinct zeros $\beta_1<\beta_2<\beta_3$ and $\beta_i\neq \pm \alpha_k$ for any $i,k$.

Indeed, the coefficients of $R_2(A)$ are real and the discriminant of $R_2(A)=0$	is	
	\begin{equation}\nonumber
			\begin{aligned}
				\Delta&=-4(-12\wp(2p))^3-27(8\wp'(2p))^2
			=1728(4\wp(2p)^3-\wp'(2p)^2)\\
				&=1728(g_2\wp(2p)+g_3)
				=1728(g_2(\wp(2p)-e_1)+4e_1^3)>0,
			\end{aligned}
		\end{equation}
		where we have used $g_2>0$ and $\wp(2p)>e_1>0$ because $2p\in (0,\frac12)$.
		This implies that $R_2(A)$ has three distinct real zeros $\beta_1<\beta_2<\beta_3$. To prove $\beta_i\neq\pm\alpha_k$, 
		we note that
		\begin{equation}\nonumber
			R_2(A)=A^3-(\alpha_1^2+\alpha_2^2+\alpha_3^2)A-2\alpha_1\alpha_2\alpha_3.
		\end{equation}
		Therefore,
		\begin{equation}\nonumber
			\begin{aligned}
				R_2(\alpha_i)&=\alpha_i^3-(\alpha_i^2+\alpha_j^2+\alpha_k^2)\alpha_i-2\alpha_i\alpha_j\alpha_k\\
				&=-\alpha_i(\alpha_j+\alpha_k)^2\neq0,
			\end{aligned}
		\end{equation}
		and similarly, 
		\begin{equation}\nonumber
			R_2(-\alpha_i)=\alpha_i(\alpha_j-\alpha_k)^2\neq0.
		\end{equation}
		Summing up the above, we have proved that $Q(A)$ has exactly nine distinct real zeros.

{\bf Step 4.} We prove \eqref{inequality}. Recall that $\alpha_1<\alpha_3<\alpha_2$ and
		\begin{equation}\nonumber
			\alpha_1\alpha_2\alpha_3=-4\wp'(2p)>0,
		\end{equation}
		thus there are two possibilities: $0<\alpha_1<\alpha_3<\alpha_2$ or $\alpha_1<\alpha_3<0<\alpha_2$. Since
		\begin{equation}\nonumber
			\alpha_1+\alpha_3=C-\alpha_2=-t_2=-\frac{\wp'(p)}{\wp(p)-e_2}>0,
		\end{equation}
		we get $0<\alpha_1<\alpha_3<\alpha_2$. Thus, $R_2(\alpha_i)<0$ and $R_2(-\alpha_i)>0$ for any $i=1,2,3$. Since $R_2(0)=8\wp'(2p)<0$, $R_2(A)\to-\infty$ as $A\to -\infty$ and $R_2(A)\to +\infty$ as $A \to +\infty$, we finally obtain
		\begin{equation}\nonumber
			\beta_1<-\alpha_2<-\alpha_3<-\alpha_1<\beta_2<0<\alpha_1<\alpha_3<\alpha_2<\beta_3.
		\end{equation}		
The proof is complete.
	\end{proof}
	
	\subsection{The structure of $\sigma_2$}
	Now we want to determine $\sigma_2$ explicitly.
	
	\begin{theorem}\label{sigma 2 is real}
		The conditional stability set $\sigma_2$ lies on the real axis, i.e., $\sigma_2\subset\R$.
	\end{theorem}
	\begin{proof}
		Take $A\in\sigma_2$, we have $r=r(A)\in\R$ and hence 
		\begin{equation}\nonumber
			\left|e^{2\pi i r}\right|=e^{-2\pi \operatorname{Im}r}=1.
		\end{equation}
		Recalling $y_1(z)=y_{a_1,a_2}(z)$ in Lemma \ref{information}, we define
		\begin{equation}\nonumber
			Y(t)\coloneqq y_1\left(\frac{1}{2}+it\right),\quad X(t)\coloneqq y_1(it),\quad 0\leq t\leq b.
		\end{equation}
		From the definition of $y_1(z)$, we know both $X(t)$ and $Y(t)$ are smooth. Since $y_1(z+\tau)=e^{2\pi i r}y_1(z)$, we know
		\begin{equation}\nonumber
			Y(b)=e^{2\pi i r}Y(0),\quad X(b)=e^{2\pi i r}X(0).
		\end{equation}
		Using the chain's rule, we know
		\begin{equation}\nonumber
			Y'(t)=iy_1'\left(\frac{1}{2}+it\right),\quad X'(t)=iy_1'(it),
		\end{equation}
		and similarly,
		\begin{equation}\nonumber
			Y'(b)=e^{2\pi i r}Y'(0),\quad X'(b)=e^{2\pi i r}X'(0).
		\end{equation}
		Again, by the chain's rule, we have
		\begin{equation}\nonumber
			Y''(t)=-y_1''\left(\frac{1}{2}+it\right)=-I\left(\frac{1}{2}+it\right)Y(t).
		\end{equation}
		Integrating from $0$ to $b$ yields that
		\begin{equation}\nonumber
			\int_0^b Y''(t)\overline{Y(t)}\,dt=-\int_0^b I\left(\frac{1}{2}+it\right)|Y(t)|^2\,dt.
		\end{equation}
		Since
		\begin{equation}\nonumber
			\overline{Y(b)}=e^{-2\pi i r}\overline{Y(0)},
		\end{equation}
		the integration by part gives
		\begin{equation}\label{integral y}
			\int_0^b I\left(\frac{1}{2}+it\right)|Y(t)|^2\,dt=\int_0^b |Y'(t)|^2\,dt\in\R.
		\end{equation}
		Similarly, one get
		\begin{equation}\label{integral x}
			\int_0^b I(it)|X(t)|^2\,dt=\int_0^b |X'(t)|^2\,dt\in\R.
		\end{equation}
		For simplicity, denote
		\begin{equation}\nonumber
			V(z)=2(\wp(z+p)+\wp(z-p))-\wp(2p)
		\end{equation}
		and
		\begin{equation}\nonumber
			D(z)=\zeta(z+p)-\zeta(z-p)-\zeta(2p),
		\end{equation}
		then
		\begin{equation}\nonumber
			I(z)=V(z)+AD(z)+\frac{A^2}{4}-\frac{2\wp'(2p)}{A}.
		\end{equation}
		Since $p\in (0, \frac14)$,  It is easy to see from \eqref{eqwpz} that on the straight lines $z=i\R$ and $z=\frac{1}{2}+i\R$, there holds $V(z)\in\R$ and $D(z)\in \R$.

	Assume by contradiction that $A=\alpha+i\beta$ with $\beta\neq0$, then
		\begin{equation}\nonumber
			\operatorname{Im}\frac{A^2}{4}=\frac{\alpha\beta}{2}
		\end{equation}
		and
		\begin{equation}\nonumber
			-\frac{2\wp'(2p)}{A}=-\frac{2\wp'(2p)(\alpha-i\beta)}{|A|^2}.
		\end{equation}
		Therefore, on the straight lines $z=i\R$ and $z=\frac{1}{2}+i\R$, there holds
		\begin{equation}\nonumber
			\operatorname{Im}I(z)=\beta D(z)+\frac{\alpha\beta}{2}+\frac{2\wp'(2p)\beta}{|A|^2}=\beta\left(D(z)+\frac{\alpha}{2}+\frac{2\wp'(2p)}{|A|^2}\right).
		\end{equation}
		Taking the imaginary part on both sides of \eqref{integral y} simultaneously yields that
		\begin{equation}\nonumber
			0=\beta\int_0^b\left(D\left(\frac{1}{2}+it\right)+\frac{\alpha}{2}+\frac{2\wp'(2p)}{|A|^2}\right)|Y(t)|^2\,dt,
		\end{equation}
		from which we have
		\begin{equation}\nonumber
			\frac{\int_0^b D\left(\frac{1}{2}+it\right)|Y(t)|^2\,dt}{\int_0^b |Y(t)|^2\,dt}=-\frac{\alpha}{2}-\frac{2\wp'(2p)}{|A|^2}.
		\end{equation}
		By a same argument, we have
		\begin{equation}\nonumber
			\frac{\int_0^b D(it)|X(t)|^2\,dt}{\int_0^b |X(t)|^2\,dt}=-\frac{\alpha}{2}-\frac{2\wp'(2p)}{|A|^2}.
		\end{equation}
		Therefore, we know
		\begin{equation}\label{contradiction point}
			\frac{\int_0^b D(it)|X(t)|^2\,dt}{\int_0^b |X(t)|^2\,dt}=\frac{\int_0^b D\left(\frac{1}{2}+it\right)|Y(t)|^2\,dt}{\int_0^b |Y(t)|^2\,dt}.
		\end{equation}
		Using the addition formula \eqref{add2}, we know
		\begin{equation}\nonumber
			D(z)=2\zeta(p)-\zeta(2p)-\frac{\wp'(p)}{\wp(z)-\wp(p)}.
		\end{equation}
		Define
		\begin{equation}\nonumber
			H(u)\coloneqq 2\zeta(p)-\zeta(2p)-\frac{\wp'(p)}{u-\wp(p)},
		\end{equation}
		then $D(z)=H(\wp(z))$ and
		\begin{equation}\nonumber
			H'(u)=\frac{\wp'(p)}{(u-\wp(p))^2}<0,
		\end{equation}
		so $H(u)$ is strictly decreasing for $u<\wp(p)$.
		Notice that 
		\begin{equation}\nonumber
			\wp(it)\leq e_2<e_3\leq \wp\left(\frac{1}{2}+it\right)\leq e_1<\wp(p),\quad 0<t< b,
		\end{equation}
		then
		\begin{equation}\nonumber
			D(it)=H(\wp(it))\geq H(e_2),\quad 0<t<b,
		\end{equation}
		\begin{equation}\nonumber
			D\left(\frac{1}{2}+it\right)=H\left(\wp\left(\frac{1}{2}+it\right)\right)\leq H(e_3),\quad 0\leq t\leq b,
		\end{equation}
and $e_2<e_3<\wp(p)$ implies $H(e_2)>H(e_3)$.
		Therefore,
		\begin{equation}\nonumber
			\begin{aligned}
				\frac{\int_0^b D(it)|X(t)|^2\,dt}{\int_0^b |X(t)|^2\,dt}&\geq H(e_2)>H(e_3)\geq\frac{\int_0^b D\left(\frac{1}{2}+it\right)|Y(t)|^2\,dt}{\int_0^b |Y(t)|^2\,dt},
			\end{aligned}
		\end{equation}
		contradicting to \eqref{contradiction point}. This proves $A\in\mathbb R$.
	\end{proof}

	From now on, we will conduct a more detailed analysis of $\sigma_2$ and gradually reveal its structure.
	\begin{lemma}\label{lemma5-3}
		Suppose $A\in\R\backslash\{0\}$ satisfies $Q(A)>0$, then $A\in\sigma_2$.
	\end{lemma}
	\begin{proof}
		Recall that $\Phi_e(z)=y_1(z)y_2(z)$, $W=y_1y_2'-y_1'y_2$ and 
		\begin{equation}\nonumber
			Q(A) =\Phi_e^\prime(z)^2-2\Phi_e^{\prime\prime}(z)\Phi_e(z) + 4I(z)\Phi_e(z)^2.
		\end{equation}
		Since $A\in\R$ and $p\in (0, 1/4)$, it is easy to see from \eqref{Phie exact expression} that $\Phi_e(z)\in\R$ for $z=\frac{1}{2}+it$, $t\in\R$. Define
		\begin{equation}\nonumber
			\psi(t)\coloneqq\Phi_e\left(\frac{1}{2}+it\right)\in\R,\quad t\in \R,
		\end{equation}
		then
		$$
		\R\ni\psi'(t)=i\Phi_e'\left(\frac{1}{2}+it\right).
		$$
		If there exists a point $z_*=\frac{1}{2}+it_*$ such that $\Phi_e(z_*)=0$, then
		\begin{equation}\nonumber
			0<Q(A)=\Phi_e'(z_*)^2=-\psi'(t_*)^2\leq0,
		\end{equation}
		which is a contradiction. Thus, $\Phi_e(z)\neq 0$ on the straight line $z=\frac{1}{2}+i\R$. Take $Y(t)\coloneqq y_1\left(\frac{1}{2}+it\right)$, then $Y'(t)=iy_1'\left(\frac{1}{2}+it\right)$. Recall that
		\begin{equation}\nonumber
			\frac{y_1'}{y_1}=\frac{\Phi_e'-W}{2\Phi_e},
		\end{equation}
		we have
		\begin{equation}\nonumber
			\frac{Y'(t)}{Y(t)}=i\frac{-i\psi'(t)-W}{2\psi(t)}=\frac{\psi'(t)}{2\psi(t)}-\frac{iW}{2\psi(t)}.
		\end{equation}
		Since $W^2=Q(A)>0$, we know $W\in\R\backslash\{0\}$, from which we get
		\begin{equation}\nonumber
			\frac{d}{dt}\log|Y(t)|=\frac{\psi'(t)}{2\psi(t)}.
		\end{equation}
		Integrating on $[0,b]$, we have
		\begin{equation}\nonumber
			\log\frac{|Y(b)|}{|Y(0)|}=\frac{1}{2}\log\frac{|\psi(b)|}{|\psi(0)|}.
		\end{equation}
		Since $\Phi_e$ is even and elliptic, we know
		\begin{equation}\nonumber
			\psi(b)=\Phi_e\left(\frac{1}{2}+ib\right)=\Phi_e\left(\frac{1}{2}\right)=\psi(0).
		\end{equation}
		Then, we know $|Y(b)|=|Y(0)|$. Notice that $Y(b)=e^{2\pi i r(A)}Y(0)$, we get $\operatorname{Im} r(A)=0$, which means $A\in\sigma_2$.
	\end{proof}
	
	\begin{lemma}\label{negative Q outside sigma2}
Suppose that $A\in\R\setminus\{0\}$ and $Q(A)<0$.
Then $A\notin\sigma_2$.
\end{lemma}
\begin{proof}
Fix $A_0\in\R\setminus\{0\}$ with $Q(A_0)<0$.
We can choose $\varepsilon>0$ sufficiently small such that
\[
\begin{gathered}
    0\notin D\coloneqq\{A\in\C:|A-A_0|<\varepsilon\},\quad
    Q(A)<0,\quad\forall A\in D\cap\R,
\end{gathered}
\]
and $e^{2\pi i r(A)}$ is analytic for $A\in D$.
Choose a holomorphic square root $W$ on $D$ satisfying
$W(A)^2=Q(A)$ and $W(A_0)=i\sqrt{-Q(A_0)}.$ 
By the continuity, we have
\[
    W(A)=i\kappa(A),\quad \kappa(A)>0,
    \quad A\in D\cap\R.
\]

Let $A\in D\cap \mathbb R$.
In the following content, we write $\Phi_e(z)=\Phi_e(z;A)$ and $y_1(z)=y_1(z;A)$ to emphasize the dependence on $A$. Recall from Section \ref{section 2} that, since $Q(A)<0$, we have $\{a_1, a_2\}\cap\{-a_1, -a_2\}=\emptyset$ in $E_{\tau}$, so $\pm a_j\neq \frac12$ in $E_{\tau}$, namely $\Phi_e(\frac12; A)\neq 0$.
As in the proof of Lemma \ref{lemma5-3}, define
\begin{equation}\nonumber
			\psi(t;A)\coloneqq\Phi_e\left(\frac{1}{2}+it;A\right)\in\R,\quad Y(t;A)\coloneqq y_1\left(\frac{1}{2}+it;A\right),\quad t\in \R,
		\end{equation}
then $\psi(0; A)\neq 0$ and 
		\begin{equation}\nonumber
			\frac{Y'(0;A)}{Y(0;A)}=\frac{\psi'(0;A)-iW(A)}{2\psi(0;A)}=\frac{\psi'(0;A)+\kappa(A)}{2\psi(0;A)}\in \mathbb R.
		\end{equation}
Since $A\in D\cap\R$ and $p\in (0, 1/4)$, it is easy to see that $I(\frac12+it; A)\in\R$ for $t\in\R$. Then $y(t)\coloneqq\frac{Y(t;A)}{Y(0;A)}$ is a solution of
$$\begin{cases}y''(t)=-I(\frac12+it; A)y(t),\quad t\in\mathbb R\\
y(0)=1,\quad y'(0)\in\mathbb{R},\end{cases}$$
so $y(t)=\frac{Y(t;A)}{Y(0;A)}\in\mathbb R$ for all $t\in\mathbb R$. Consequently,
$$e^{2\pi i r(A)}=\frac{y_1(\frac12+ib; A)}{y_1(\frac12;A)}=\frac{Y(b; A)}{Y(0; A)}\in\mathbb{R}\setminus\{0\},\quad \forall A\in D\cap \mathbb R.$$

Suppose to the contrary that $A_0\in\sigma_2$. Then $|e^{2\pi i r(A_0)}|=1$ and so
\[
   e^{2\pi i r(A_0)}=\pm 1.
\]
If $e^{2\pi i r(A)}$ is constant for $A\in D$, we have $e^{2\pi i r(A)}=\pm 1$ and so $A\in \sigma_2$ for all $A\in D$,
which contradicts to the fact that $\sigma_2\subset\mathbb R$.
Thus $e^{2\pi i r(A)}$ is not a constant for $A\in D$. Since $e^{2\pi i r(A)}$ is analytic for $A\in D$, the open mapping theorem implies that there exists a constant
 $\delta>0$ such that
\[
    \{\mu\in\C:|\mu-e^{2\pi i r(A_0)}|<\delta\}
    \subset\{e^{2\pi i r(A)} : A\in D\}.
\]
Since $e^{2\pi i r(A_0)}=\pm 1$, there exists a constant $\mu\in\mathbb{C}\setminus\mathbb R$ such that $|\mu|=1$ and $|\mu-e^{2\pi i r(A_0)}|<\delta$. Then there is $\tilde A\in D$ such that
$e^{2\pi i r(\tilde A)}=\mu$. However,
\[
\begin{aligned}
    e^{2\pi i r(\tilde A)}\notin\R
    &\quad\Longrightarrow\quad \tilde A\notin\R,\\
    |e^{2\pi i r(\tilde A)}|=1
    &\quad\Longrightarrow\quad \tilde A\in\sigma_2\subset\R.
\end{aligned}
\]
Hence, we have $A_0\notin\sigma_2$. The proof is complete.
\end{proof}

Denote $$\mathcal{Z}=\{A : Q(A)=0\}=\{\beta_1,\beta_2,\beta_3,\pm\alpha_1,\pm\alpha_2,\pm\alpha_3\}.$$
Then the argument in Section \ref{section 2} gives
$\mathcal Z\subset\sigma_1\cap\sigma_2$.
Together with the preceding two lemmas, this proves the following corollary.

\begin{corollary}\label{six intervals}
		$\sigma_2=\{A\in\R\backslash\{0\}:Q(A)\geq0\}$. More concretely, $\sigma_2=\mathcal Z\cup\bigcup_{j=1}^6 I_j$, where
		\begin{equation}\nonumber
			\begin{aligned}
				&I_1=(-\infty,\beta_1),\quad I_2=(-\alpha_2,-\alpha_3),\quad I_3=(-\alpha_1,\beta_2),\\
				&I_4=(0,\alpha_1),\quad I_5=(\alpha_3,\alpha_2),\quad I_6=(\beta_3,+\infty).
			\end{aligned}
		\end{equation}
	\end{corollary}
	
		\subsection{The structure of $\sigma_1\cap\sigma_2$}
		Remark that we have no idea how to determine $\sigma_1$ explicitly.
		Here, we turn to determine $\sigma_1\cap\sigma_2$, where is equivalent to 
 determining whether the intervals $I_1,\cdots,I_6$ contain any elements of $\sigma_1$.
	
	Recall that
	\begin{equation}\nonumber
		y_1(z+1)=e^{-2\pi i s(A)}y_1(z),\quad y_2(z+1)=e^{2\pi i s(A)}y_2(z).
	\end{equation}
	Define
	\begin{equation}\nonumber
		u(A)\coloneqq\log|e^{-2\pi is(A)}|,
	\end{equation}
	then $A\in\sigma_1$ is equivalent to $u(A)=0$. Fix $A\in I_j\subset\sigma_2$, $j=1,\cdots,6$. Define
	\begin{equation}\nonumber
		\phi(x)\coloneqq\Phi_e\left(x+\frac{\tau}{2}\right),\quad x\in\mathbb R.
	\end{equation}
	Since $\bar{z}=z-\tau$ for any $z=x+\frac{\tau}{2}$, $x\in\R$, we have $\phi(x)\in\R$ is smooth and 
	\begin{equation}\nonumber
		\phi'(x)=\Phi_e'\left(x+\frac{\tau}{2}\right)\in\R.
	\end{equation}
	Besides, since $W^2=Q(A)>0$, we know $W\in\R\backslash\{0\}$. Take $x_0\in\mathbb R$ such that $\phi(x_0)\neq0$. Then, $y_1(x_0+\frac{\tau}{2})\neq 0$ and $y_2(x_0+\frac{\tau}{2})\neq 0$. Denote
	\begin{equation}\nonumber
		Y(x)\coloneqq y_1\left(x+\frac{\tau}{2}\right),\quad x\in\mathbb R,
	\end{equation}
	then
	\begin{equation}\nonumber
		\frac{Y'(x_0)}{Y(x_0)}=\frac{y_1'\left(x_0+\frac{\tau}{2}\right)}{y_1\left(x_0+\frac{\tau}{2}\right)}=\frac{\Phi_e'\left(x_0+\frac{\tau}{2}\right)-W}{2\Phi_e\left(x_0+\frac{\tau}{2}\right)}=\frac{\phi'(x_0)-W}{2\phi(x_0)}\in \R,
	\end{equation}
	and
	$$Y''(x)=I\left(x+\frac{\tau}{2}; A\right)Y(x),\quad x\in\mathbb{R}.$$
	Since $A\in I_j\subset\R$ and $p\in (0, 1/4)$, it is easy to see $I(x+\frac{\tau}{2}; A)\in\R$ for $x\in\R$. Then the same proof as Lemma \ref{negative Q outside sigma2} implies $\frac{Y(x)}{Y(x_0)}\in\mathbb R$ for all $x\in\mathbb R$, so
	\begin{equation}
		e^{-2\pi i s(A)}=\frac{Y(x_0+1)}{Y(x_0)}\in\R\backslash\{0\},\quad \forall A\in I_j,\;\;j=1,\cdots,6.
	\end{equation}
	
	We first rule out the possibility that $I_1,I_4,I_6$ contain any elements of $\sigma_1$.
	\begin{lemma}\label{rule out 3}
		$I_1\cap\sigma_1=I_4\cap\sigma_1=I_6\cap\sigma_1=\emptyset$.
	\end{lemma}
	\begin{proof}
		We see from \eqref{Phie exact expression} that
		\begin{equation}\nonumber
			\phi(x)=\Phi_e\left(x+\frac{\tau}{2}\right)=\frac{A^2}{2}+AV_1(x)+V_0(x)+\frac{2\wp'(2p)}{A},
		\end{equation}
		where
		\begin{equation}\nonumber
			V_0(x)=\wp\left(x+\frac{\tau}{2}+p\right)+\wp\left(x+\frac{\tau}{2}-p\right)-2\wp(2p)
		\end{equation}
		and
		\begin{equation}\nonumber
			V_1(x)=\zeta(2p)-\zeta\left(x+\frac{\tau}{2}+p\right)+\zeta\left(x+\frac{\tau}{2}-p\right).
		\end{equation}
		In particular, both $V_0$ and $V_1$ are $1$-period, real and smooth in $\R$. Denote $$K_0\coloneqq\max_{x\in\R}|V_0(x)|,\quad K_1\coloneqq\max_{x\in\R}|V_1(x)|,$$ and
		$$m(A)\coloneqq\min_{x\in\R}\phi(x),\quad M(A)\coloneqq\max_{x\in\R}\phi(x).$$ Obviously, $m(A)$ and $M(A)$ are continuous with respect to $A$. We now prove that for $A\in I_1\cup I_4\cup I_6$, both $m(A)$ and $M(A)$ can not be $0$. If $m(A)=0$, then there exists $x_0\in\R$ such that $\phi(x_0)=0$ and $\phi'(x_0)=0$. Then,
		\begin{equation}\nonumber
			Q(A)=\Phi_e'^2-2\Phi_e''\Phi_e+4I\Phi_e^2=0,
		\end{equation}
		which is a contradiction. Similarly, one can prove that $M(A)\neq0$. Therefore, for all $A\in I_j$, $j\in{1,4,6}$, both $\operatorname{sgn}M(A)$ and $\operatorname{sgn}m(A)$ are invariant. When $|A|$ is sufficiently large, we know
		\begin{equation}\nonumber
			\phi(x)\geq\frac{|A|^2}{2}-K_1|A|-K_0-\frac{2|\wp'(2p)|}{|A|}>0,\quad\forall x\in\R.
		\end{equation}
		Then, by the continuity property, we get
		\begin{equation}\nonumber
			\phi(x)>0,\quad\forall x\in\R,\quad \forall A\in I_1\cup I_6.
		\end{equation}
	Similarly, when $A\to 0_+$, it follows from $\wp'(2p)<0$ that $\phi(x)<0$ for all $x\in\mathbb R$, hence
		\begin{equation}\nonumber
			\phi(x)<0,\quad\forall x\in\R,\quad \forall A\in I_4.
		\end{equation}
		Recall that $y_1y_2=\Phi_e$, so $y_1$ has no zeros on the line $\R+\frac{\tau}{2}$, namely $Y(x)=y_1(x+\frac{\tau}{2})$ has no zeros for $x\in\mathbb R$. Then for any $A\in I_1\cup I_4\cup I_6$, it follows from $\frac{Y'(x)}{Y(x)}=\frac{\phi'(x)-W}{2\phi(x)}\in\mathbb R$ that
		\begin{equation}\nonumber
			\begin{aligned}
				u(A)&=\log|e^{-2\pi i s(A)}|=\int_0^1\frac{d}{dx}\log |Y(x)|\,dx\\
				&=\int_0^1\frac{Y'(x)}{Y(x)}\,dx=\int_0^1\frac{\phi'(x)-W}{2\phi(x)}\,dx\\
				&=\frac{1}{2}(\log|\phi(1)|-\log|\phi(0)|)-\int_0^1\frac{W}{2\phi(x)}\,dx\\
				&=-W\int_0^1\frac{1}{2\phi(x)}\,dx\neq0.
			\end{aligned}
		\end{equation}
		Therefore, $s(A)$ can not be real and $A\not\in\sigma_1$.
	\end{proof}
	From now on, we consider the remaining intervals $I_2=(-\alpha_2,-\alpha_3)$, $I_3=(-\alpha_1,\beta_2)$ and $I_5=(\alpha_3,\alpha_2)$. The following polynomial will plays a crucial role:
	\begin{equation}\label{polynomial N}
		N(A)\coloneqq A^6+4(\eta_1-2\wp(2p))A^4+4\wp'(2p)A^3-16\wp'(2p)(\eta_1+\wp(2p))A+16\wp'(2p)^2.
	\end{equation}
	\begin{lemma}\label{lemma57}
		The polynomial $N(A)$ has exactly two simple negative zeros and at most two positive zeros, counting multiplicities.
	\end{lemma}
	\begin{proof}
		Let
		\begin{equation}\nonumber
			k=\left(-\frac{\wp'(2p)}{2}\right)^{\frac{1}{3}}>0,\quad C_*=\frac{3\wp(2p)}{k^2},\quad E_*=\frac{\eta_1+\wp(2p)}{k^2}.
		\end{equation}
		Since $e_1\neq e_2\neq e_3\neq e_1$ and $\wp(2p)>e_k$, by the inequality of arithmetic and geometric means, we have
		\begin{equation}\nonumber
			k^6=\frac{\wp'(2p)^2}{4}=\prod_{j=1}^3(\wp(2p)-e_j)<\left(\frac{3\wp(2p)-e_1-e_2-e_3}{3}\right)^3=\wp(2p)^3,
		\end{equation}
		from which we know $C_*>3$. First of all, we study the existence of negative zeros. Let $A=-2kt$, $t>0$, then
		\begin{equation}\nonumber
			\begin{aligned}
				N(-2kt)=&64k^6t^6+64(\eta_1-2\wp(2p))k^4t^4-32\wp'(2p)k^3t^3\\
				&+32\wp'(2p)(\eta_1+\wp(2p))kt+16\wp'(2p)^2\\
				=&64k^6[t^6+(E_*-C_*)t^4+t^3-E_*t+1]\\
				\eqqcolon &64k^6t^4[\mathcal{G}_n(t)-C_*],
			\end{aligned}
		\end{equation}
		where
		\begin{equation}\nonumber
			\mathcal{G}_n(t)\coloneqq t^2+t^{-1}+t^{-4}+E_*(1-t^{-3}).
		\end{equation}
		Therefore, finding a negative zero of $N(A)$ is equivalent to finding a solution of $\mathcal{G}_n(t)=C_*$. By a basic computation, we have
		\begin{equation}\nonumber
			\mathcal{G}_n'(t)=2t-\frac{1}{t^2}-\frac{4}{t^5}+\frac{3E_*}{t^4}=t^{-4}\left(2t^5-t^2+3E_*-\frac{4}{t}\right)\eqqcolon t^{-4}\mathcal{T}_n(t),
		\end{equation}
		and
		\begin{equation}\nonumber
			\mathcal{T}_n'(t)=10t^4-2t+\frac{4}{t^2}=\frac{2}{t^2}(5t^6-t^3+2)=\frac{2}{t^2}\left(5\left(t^3-\frac{1}{10}\right)^2+\frac{39}{20}\right)>0.
		\end{equation}
		Thus, $\mathcal{T}_n(t)$ is strictly increasing. Moreover, since
		\begin{equation}\nonumber
			\lim_{t\to 0_+} \mathcal{T}_n(t)=-\infty,\quad\lim_{t\to +\infty}\mathcal{T}_n(t)=+\infty,
		\end{equation}
		we know there exists $t_*\in\R_+$ such that $t_*$ is a minimum point of $\mathcal{G}_n(t)$. Moreover, $\mathcal{G}_n(t)$ is decreasing in $(0,t_*)$ and increasing in $(t_*,+\infty)$. Since $\mathcal{G}_n(1)=3<C_*$, we know the equation $\mathcal{G}_n(t)=C_*$ has exactly two solutions $t_1,t_2$ satisfying that $t_1\in(0,1)$ and $t_2\in(1,+\infty)$. Both $t_1$ and $t_2$ are not minimum point, and hence must be simple. This proves that $N(A)$ has exactly two simple negative zeros.
		
		Now, we study the existence of positive zeros. Let $A=2kt$, $t>0$, then
		\begin{equation}\nonumber
			N(2kt)=64k^6t^4[\mathcal{G}_p(t)-C_*],
		\end{equation}
		where
		\begin{equation}\nonumber
			\mathcal{G}_p(t)\coloneqq t^2-t^{-1}+t^{-4}+E_*(1+t^{-3}).
		\end{equation}
		Similarly, 
		\begin{equation}\nonumber
			\mathcal{G}_p'(t)=2t+\frac{1}{t^2}-\frac{4}{t^5}-\frac{3E_*}{t^4}=t^{-4}\left(2t^5+t^2-3E_*-\frac{4}{t}\right)\eqqcolon t^{-4}\mathcal{T}_p(t),
		\end{equation}
		and
		\begin{equation}\nonumber
			\mathcal{T}_p'(t)=10t^4+2t+\frac{4}{t^2}>0.
		\end{equation}
		Since 
		\begin{equation}\nonumber
			\lim_{t\to 0_+} \mathcal{T}_p(t)=-\infty,\quad\lim_{t\to +\infty}\mathcal{T}_p(t)=+\infty,
		\end{equation}
		we know $\mathcal{G}_p$ firstly strictly decreasing, and then strictly increasing, as $t$ varies from $0$ to $+\infty$. So $\mathcal{G}_p-C_*=0$ has at most two zeros on $(0,+\infty)$, namely $N(A)$ has at most two positive zeros. Suppose that $t_0\in(0,+\infty)$ is the minimum point of $\mathcal{G}_p$, then
		\begin{equation}\nonumber
			\mathcal{G}_p''(t_0)=t_0^{-4}\mathcal{T}_p'(t_0)>0.
		\end{equation}
Consequently, if $\mathcal{G}_p(t_0)=C_*$, then $2kt_0$ is a positive zero of multiplicity $2$. If $\mathcal{G}_p(t_0)<C_*$, then $N(A)$ has two positive simple zeros, and if $\mathcal{G}_p(t_0)>C_*$, then $N(A)$ has no positive zero. The proof is complete.
	\end{proof}
	
	Now, we start to find the connection between $u(A)$ and $N(A)$. 
Let $A\in I_j$ for $j\in \{2,3,5\}$, then $Q(A)>0$, which implies $\{a_1, a_2\}\cap\{-a_1, -a_2\}=\emptyset$ in $E_{\tau}$, so $\pm a_j\neq 0$ in $E_{\tau}$, namely $\Phi_e(0; A)\neq 0$ for all $A\in I_j$. Therefore, we can choose analytic square roots $W(A)=\sqrt{Q(A)}>0$ and $\sqrt{\Phi_e(0; A)}$ for $A\in I_j$.

Note from the evenness of $\Phi_e$ that $\Phi_e'(0; A)=0$.
Recalling \eqref{eqy1y2}, we 
	let $\tilde y_1$ satisfy GLE \eqref{generalized lame equation} with initial conditions
	\begin{equation}\nonumber
		\tilde y_1(0;A)=\sqrt{\Phi_e(0;A)},\quad \tilde y_1'(0;A)=-\frac{W(A)}{2\sqrt{\Phi_e(0;A)}},
	\end{equation}
	and define $\tilde y_2(z;A)\coloneqq\tilde y_1(-z;A)$. Then $\tilde y_1(z; A)$ and $\tilde y_2(z; A)$ are analytic for $A\in I_j$. Note that the Wronskain
	$$\tilde y_1\tilde y_2'-\tilde y_1'\tilde y_2\equiv (\tilde y_1\tilde y_2'-\tilde y_1'\tilde y_2)(0)=W(A)>0.$$
Furthermore, $\tilde y_1\tilde y_2$ is a solution of the third order ODE \eqref{second symmetric product equation}. By \eqref{eqQA}, it is easy to verify that
	\begin{equation}\nonumber
		(\tilde y_1\tilde y_2)(0)=\Phi_e(0;A),\quad (\tilde y_1\tilde y_2)'(0)=0=\Phi_e'(0;A),\quad (\tilde y_1\tilde y_2)''(0)=\Phi_e''(0;A).
	\end{equation}
	Therefore, $\tilde y_1\tilde y_2=\Phi_e=y_1y_2$, which implies that there is $k\in\{1,2\}$ such that $\tilde y_1= cy_k$ for some $c\neq 0$.  Then
	$$\tilde y_1(z+1;A)=\varepsilon_1(A)\tilde y_1(z;A),\quad\text{where }\varepsilon_1(A)\coloneqq e^{(-1)^k2\pi i s(A)}.$$ 
	Consequently, $\varepsilon_1(A)$ is analytic for $A\in I_j$.
	Note that
	$$u(A)=\log|e^{-2\pi i s(A)}|=(-1)^{k-1}\log|\varepsilon_1(A)|,$$
	and we want to study whether $u(A)=0$ for some $A\in I_j$,
	hence, by replacing $u(A)$ with $-u(A)$ if necessary, we can always take $k=1$ in the sequel, i.e.,
$$\varepsilon_1(A)=e^{-2\pi i s(A)},\quad u(A)=\log|\varepsilon_1(A)|.$$	
	 Then we have the following result.
	
	\begin{lemma}
	Under the above notations,  we have
		\begin{equation}\label{derivative}
		u'(A)=-\frac{N(A)}{4A^3W(A)}=-\frac{N(A)}{4A^3\sqrt{Q(A)}},\quad\forall A\in I_j.
	\end{equation}
	\end{lemma}

\begin{proof}	
Since $\tilde y_1''=I\tilde y_1$, we have
	\begin{equation}\nonumber
		\partial_A\tilde y_1''=\tilde y_1\partial_AI+I\partial_A \tilde y_1.
	\end{equation}
	Denote
	\begin{equation}\nonumber
		B(z)=B(z;A)\coloneqq\tilde y_2(\partial_A\tilde y_1)'-\tilde y_2'\partial_A\tilde y_1,
	\end{equation}
	then
	\begin{equation}\nonumber
		\begin{aligned}
			B'&=\tilde y_2(\partial_A\tilde y_1)''-\tilde y_2''\partial_A\tilde y_1
			=\tilde y_2(\tilde y_1\partial_AI+I\partial_A\tilde y_1)-I\tilde y_2\partial_A\tilde y_1\\
			&=\tilde y_1\tilde y_2\partial_AI=\Phi_e\partial_AI.
		\end{aligned}
	\end{equation}
	Taking $z_0=\frac{\tau}{2}$ and integrating from $z_0$ to $z_0+1$ yields that
	\begin{equation}\nonumber
		B(z_0+1)-B(z_0)=\int_{z_0}^{z_0+1}\Phi_e \partial_AI\,dz.
	\end{equation}
	Since $\tilde y_1(z+1;A)=\varepsilon_1(A)\tilde y_1(z;A)$, we have $$\partial_A \tilde y_1(z_0+1;A)=\partial_A\varepsilon_1(A) \tilde y_1(z_0;A)+\varepsilon_1(A)\partial_A\tilde y_1(z_0;A),$$ $$(\partial_A\tilde y_1)'(z_0+1;A)=\partial_A\varepsilon_1(A) \tilde y_1'(z_0;A)+\varepsilon_1(A)(\partial_A\tilde y_1)'(z_0;A).$$ Hence, by a direct calculation, we know
	\begin{equation}\nonumber
		\begin{aligned}
			B(z_0+1)&=\tilde y_2(z_0+1)(\partial_A\tilde y_1)'(z_0+1)-\tilde y_2'(z_0+1)\partial_A\tilde y_1(z_0+1)\\
			&=\varepsilon_1(A)^{-1}\tilde y_2(z_0)\left(\partial_A\varepsilon_1(A)\tilde y_1'(z_0)+\varepsilon_1(A)(\partial_A\tilde y_1)'(z_0)\right)\\
			&\quad-\varepsilon_1(A)^{-1}\tilde y_2'(z_0)\left(\partial_A\varepsilon_1(A) \tilde y_1(z_0)+\varepsilon_1(A)\partial_A\tilde y_1(z_0)\right)\\
			&=B(z_0)-W(A)\frac{\partial_A\varepsilon_1(A)}{\varepsilon_1(A)},
		\end{aligned}
	\end{equation}
	or equivalently,
	\begin{equation}\nonumber
		\frac{\partial_A\varepsilon_1(A)}{\varepsilon_1(A)}=-\frac{J_1(A)}{W(A)},
	\end{equation}
	where
	\begin{equation}\nonumber
		J_1(A)\coloneqq\int_{z_0}^{z_0+1}\Phi_e \partial_AI\,dz=\int_{\frac{\tau}{2}}^{\frac{\tau}{2}+1}\Phi_e \partial_AI\,dz.
	\end{equation}
	On the interval $\left(\frac{\tau}{2},\frac{\tau}{2}+1\right)$, we know $I,\partial_A I,\Phi_e, W(A)\in\R$. Therefore,
	\begin{equation}\nonumber
		u'(A)=\operatorname{Re}\frac{\partial_A\varepsilon_1(A)}{\varepsilon_1(A)}=-\frac{J_1(A)}{W(A)}.
	\end{equation}
	
It remains to compute $J_1(A)$. Define
	\begin{equation}\nonumber
		\mathcal{S}_p(z)\coloneqq\wp(z+p)+\wp(z-p),\quad \mathcal{H}_p(z)\coloneqq\zeta(z+p)-\zeta(z-p)-\zeta(2p).
	\end{equation}
	Since both $\mathcal{H}_p$ and $\mathcal{S}_p$ are both elliptic functions, by Laurent's expansion, one can verify that
	\begin{equation}\nonumber
		\mathcal{H}_p^2=\mathcal{S}_p+\wp(2p),\quad \mathcal{H}_p(\mathcal{S}_p-2\wp(2p))=\frac{1}{2}\mathcal{H}_p''-\wp'(2p).
	\end{equation}
Since $\zeta'=-\wp$ and $\zeta(z+1)-\zeta(z)=\eta_1$, we have
	\begin{equation}\nonumber
		\begin{aligned}
			\int_{\frac{\tau}{2}}^{\frac{\tau}{2}+1}\mathcal{S}_p(z)\,dz&=\int_{\frac{\tau}{2}}^{\frac{\tau}{2}+1}-(\zeta'(z+p)+\zeta'(z-p))\,dz=-2\eta_1,
		\end{aligned}
	\end{equation}
	which implies that
	\begin{equation}\nonumber
		\int_{\frac{\tau}{2}}^{\frac{\tau}{2}+1}\mathcal{H}_p^2(z)\,dz=\wp(2p)-2\eta_1.
	\end{equation}
	On the other hand, since $\mathcal{H}_p$ is elliptic, we know
	\begin{equation}\nonumber
		\int_{\frac{\tau}{2}}^{\frac{\tau}{2}+1}\mathcal{H}_p''(z)\,dz=0,
	\end{equation}
	which implies that
	\begin{equation}\nonumber
		\int_{\frac{\tau}{2}}^{\frac{\tau}{2}+1}\mathcal{H}_p(z)(\mathcal{S}_p(z)-2\wp(2p))\,dz=-\wp'(2p).
	\end{equation}
	Denote
	\begin{equation}\nonumber
		K(A)\coloneqq\frac{A}{2}+\frac{2\wp'(2p)}{A^2},\quad L(A)\coloneqq\frac{A^2}{2}+\frac{2\wp'(2p)}{A}-2\wp(2p),
	\end{equation}
	then $\partial_AI=\mathcal{H}_p+K$ and $\Phi_e=\mathcal{S}_p-A\mathcal{H}_p+L$. Since $L-AK=-2\wp(2p)$, we know
	\begin{equation}\nonumber
		\begin{aligned}
			\Phi_e\partial_AI &=(\mathcal{H}_p+K)(\mathcal{S}_p-A\mathcal{H}_p+L)\\
			&=\mathcal{H}_p\mathcal{S}_p-A\mathcal{H}_p^2+\mathcal{H}_pL+K\mathcal{S}_p-A\mathcal{H}_pK+KL\\
			&=\mathcal{H}_p(\mathcal{S}_p-2\wp(2p))-A\mathcal{H}_p^2+K\mathcal{S}_p+KL.
		\end{aligned}
	\end{equation}
	Summing up the above, we have
	\begin{equation}\nonumber
		J_1(A)=-\wp'(2p)-A(\wp(2p)-2\eta_1)-2\eta_1K+KL=\frac{N(A)}{4A^3}.
	\end{equation}
	This implies \eqref{derivative}.
\end{proof}

	Denote $$J=(l_J,r_J)\in\{I_2,I_3,I_5\},$$ we know $\Delta_1(l_J),\Delta_1(r_J)\in\{\pm 1\}$ by the argument of Case 2 in Section \ref{section 2}. Since
	$$\varepsilon_1(A)^2-2\Delta_1(A)\varepsilon_1(A)+1=0,\quad\forall A\in J,$$
	 we get $\varepsilon_1(A)\to\Delta_1(l_J)\in\{\pm 1\}$ as $A\to {l_J}_+$, and  $\varepsilon_1(A)\to\Delta_1(r_J)\in\{\pm 1\}$ as $A\to {r_J}_-$. Moreover, $$u({l_J})=u({r_J})=0.$$
	Since $l_J$ is a simple zero of $Q$ and $Q>0$ in $(l_J,r_J)$, we know
	\begin{equation}\nonumber
		\frac{Q(A)}{A-l_J}\to Q'(l_J)>0,\quad\text{as }A\to l_J.
	\end{equation}
Thus, near $l_J$, we have
	\begin{equation}\nonumber
		\left|\frac{N(t)}{4t^3\sqrt{Q(t)}}\right|\leq\frac{C}{\sqrt{t-l_J}}.
	\end{equation}
	The analysis near the right endpoint $r_J$ is similar and we can finally conclude that the integral
	$$
	\int_{l_J}^{A}\frac{N(t)}{4t^3\sqrt{Q(t)}}\,dt
	$$
	is integrable at both endpoints. 
	Integrating \eqref{derivative} gives that
	\begin{equation}\nonumber
		u(A)=-\int_{l_J}^{A}\frac{N(t)}{4t^3\sqrt{Q(t)}}\,dt,\quad \forall A\in [l_J, r_J].
	\end{equation}
Clearly $u\in C([l_J,r_J])\cap C^1((l_J,r_J))$, and
	\begin{equation}\nonumber
		A\in\sigma_1\cap J\Leftrightarrow u(A)=0,
\quad
		u'(A)=0\Leftrightarrow N(A)=0.
	\end{equation}
	
\begin{lemma}	Recall $I_2=(-\alpha_2,-\alpha_3)$, $I_3=(-\alpha_1,\beta_2)$ and $-\alpha_3<-\alpha_1<\beta_2<0$. We have
	$I_2\cap \sigma_1=I_3\cap\sigma_1=\emptyset$.
	\end{lemma}
	\begin{proof}
Let $J\in\{I_2,I_3\}$. By $u(l_J)=u(r_J)=0$, there exists $A_J\in(l_J,r_J)$ such that $u'(A_J)=0$, i.e., $N(A_{I_2})=N(A_{I_3})=0$. Since $N(A)$ has exactly two negative simple zeros, we have
\begin{equation}\label{eee}u'(A)\neq 0,\quad\forall A\in (l_J, r_J)\setminus\{A_J\}.\end{equation}
This implies $u(A)\neq 0$ for any $A\in I_2\cup I_3$. Indeed, if $A_0\in(l_J,r_J)$ and $u(A_0)=0$, then there exist $A_1\in(l_J,A_0)$ and $A_2\in(A_0,r_J)$ such that $u'(A_1)=u'(A_2)=0$, a contradiction with \eqref{eee}. Thus,  $I_2\cap \sigma_1=I_3\cap\sigma_1=\emptyset$.
	\end{proof}

\begin{lemma} Recall  $I_5=(\alpha_3, \alpha_2)$ with $\alpha_3>0$.  If $N(\alpha_2)>0$ and $N(\alpha_3)>0$, then there is a unique $A_*\in I_5$ such that $\sigma_1\cap I_5=\{A_*\}$. For all other cases, $\sigma_1\cap I_5=\emptyset$.	
\end{lemma}	
\begin{proof}
	Note
	\begin{equation}\nonumber
		0=u(\alpha_3)-u(\alpha_2)=\int_{\alpha_3}^{\alpha_2}\frac{N(A)}{4A^3\sqrt{Q(A)}}\,dA,
	\end{equation} 
	so the sign of $N(A)$ in $I_5$ must change. Take $A_0\in I_5$ such that $N(A_0)<0$. Since $N(0)>0$ and $N(A)\to+\infty$ as $A\to+\infty$, we know there exist $\xi_1\in(0,A_0)$ and $\xi_2\in(A_0,+\infty)$ such that $N(\xi_1)=N(\xi_2)=0$. 
	Then Lemma \ref{lemma57} implies
\begin{equation}\label{ece}N(A)\begin{cases}>0\quad\text{if }\; A\in (0,\xi_1)\cup (\xi_2,+\infty)\\
	<0\quad\text{if }\; A\in (\xi_1, \xi_2).\end{cases}\end{equation}

	For simplicity, denote 
	\begin{equation}\nonumber
		w(A)\coloneqq\frac{1}{4A^3\sqrt{Q(A)}}>0,\quad A\in I_5,
	\end{equation}
	then $u'(A)=-N(A)w(A)$ and $u(\alpha_2)=u(\alpha_3)=0$. Then,
	\begin{equation}\nonumber
		u(A)=-\int_{\alpha_3}^Aw(t)N(t)\,dt=\int_A^{\alpha_2}w(t)N(t)\,dt.
	\end{equation}
	There are four possibilities.
	
	First, $N(\alpha_3)>0$ and $N(\alpha_2)>0$. Since $A_0\in (\alpha_3, \alpha_2)$ and $N(A_0)<0$, so \eqref{ece} gives $\alpha_3<\xi_1<\xi_2<\alpha_2$ and
	$$
u(\xi_1)=-\int_{\alpha_3}^{\xi_1} w(t)N(t)\,dt<0,
	$$
	$$
	u(\xi_2)=\int_{\xi_2}^{\alpha_2} w(t)N(t)\,dt>0.
	$$
Also, $u'(A)=-w(A)N(A)>0$ for $A\in(\xi_1,\xi_2)$ and $u'(A)<0$ for $A\in (\alpha_3, \xi_1)\cup (\xi_2, \alpha_2)$. Together with $u(\alpha_3)=u(\alpha_2)=0$, we conclude that there exists a unique point $A_*\in(\alpha_3,\alpha_2)$ such that $u(A_*)=0$. This proves $\sigma_1\cap I_5=\{A_*\}$.
	
	Second, $N(\alpha_3)>0$ and $N(\alpha_2)\leq0$. Since $A_0\in (\alpha_3, \alpha_2)$ and $N(A_0)<0$, so \eqref{ece} gives $\alpha_3<\xi_1<\alpha_2\leq\xi_2$ and
	\begin{equation}\nonumber
		u(A)=-\int_{\alpha_3}^Aw(t)N(t)\,dt<0,\quad A\in(\alpha_3,\xi_1]
	\end{equation}
	\begin{equation}\nonumber
		u(A)=\int_{A}^{\alpha_2}w(t)N(t)\,dt<0,\quad A\in[\xi_1,\alpha_2).
	\end{equation}
	Therefore, $\sigma_1\cap I_5=\emptyset$.
	
	Third, $N(\alpha_3)\leq 0$ and $N(\alpha_2)>0$. Since $A_0\in (\alpha_3, \alpha_2)$ and $N(A_0)<0$, so \eqref{ece} gives $\xi_1\leq \alpha_3<\xi_2<\alpha_2$ and
	\begin{equation}\nonumber
		u(A)=-\int_{\alpha_3}^Aw(t)N(t)\,dt>0,\quad A\in(\alpha_3,\xi_2],
	\end{equation}
	\begin{equation}\nonumber
		u(A)=\int_{A}^{\alpha_2}w(t)N(t)\,dt>0,\quad A\in[\xi_2,\alpha_2).
	\end{equation}
	Therefore, $\sigma_1\cap I_5=\emptyset$.
	
	Fourth, $N(\alpha_3)\leq0$ and $N(\alpha_2)\leq0$. Then \eqref{ece} gives $\xi_1\leq \alpha_3<\alpha_2\leq \xi_2$ and
	\begin{equation}\nonumber
		\int_{\alpha_3}^{\alpha_2}w(t)N(t)\,dt<0,
	\end{equation}
	which is a contradiction. 
	\end{proof}
	
	Summing up the above, we have established the following result:
	\begin{theorem}\label{conditional intersection}
		Assume that $0<p<1/4$, then
		\begin{equation}
			\sigma_1\cap\sigma_2=
			\begin{cases}
				\mathcal{Z}\cup\{A_*\}, & \text{if }N(\alpha_3)>0 \text{ and }N(\alpha_2)>0; \\
				\mathcal{Z}, & \text{otherwise}.
			\end{cases}
		\end{equation}
	\end{theorem}

	We now give a necessary and sufficient condition on $p$ for the first case of Theorem \ref{conditional intersection}. Write
	\begin{equation}\label{endpoint parameters}
		D=e_1-e_2>0,\quad m=\frac{e_3-e_2}{D}\in (0,1),\quad q=\sqrt{1-m}\in (0,1).
	\end{equation}
We use the parameter $m$ in the complete elliptic integrals
	\begin{equation}\label{endpoint elliptic integrals}
		K(m)=\int_0^{\pi/2}\frac{d\theta}{\sqrt{1-m\sin^2\theta}},\quad
		E(m)=\int_0^{\pi/2}\sqrt{1-m\sin^2\theta}\,d\theta,
	\end{equation}
	and set $\kappa=1-E(m)/K(m)$.
	
	\begin{theorem}\label{endpoint positivity theorem}
		For every fixed $b>0$, there exist unique real numbers
		\begin{equation}\label{endpoint p order}
			0<p_-(b)<\frac16<p_+(b)<\frac14,
		\end{equation}
		such that, for $0<p<1/4$,
		\begin{equation}\label{endpoint iff}
			N(\alpha_2)>0\text{ and }N(\alpha_3)>0
			\quad\Longleftrightarrow\quad p\in(p_-(b),p_+(b)).
		\end{equation}
		More precisely, set
		\begin{align}
			U_2&=\frac{m}{\kappa}+\sqrt{\frac{m^2}{\kappa^2}-m},\label{endpoint U2}\\
			U_3&=m+\frac{m(1-m)}{m-\kappa}
			+\sqrt{\left(\frac{m(1-m)}{m-\kappa}\right)^2+m(1-m)}.
			\label{endpoint U3}
		\end{align}
		Then $1+q<U_3<U_2$, and the endpoints are uniquely determined in $(0,1/4)$ by
		\begin{equation}\label{endpoint inverse wp}
			\wp(p_-(b))=e_2+DU_2,\quad \wp(p_+(b))=e_2+DU_3.
		\end{equation}
	\end{theorem}
	
	\begin{proof}
		Firstly, we make some simplifications to $N(\alpha_i)$. Let $c=\wp(p)$, $y=\wp'(p)$, and, for $\{i,j,k\}=\{1,2,3\}$, denote
		\begin{equation}\nonumber
			x_i=c-e_i,\quad B_i=(e_i-e_j)(e_i-e_k)=3e_i^2-\frac{g_2}{4}.
		\end{equation}
		We claim that
		\begin{equation}\label{endpoint factor}
			N(\alpha_i)=8\alpha_i^2
			\left[(\eta_1+e_i)\left(x_i+\frac{B_i}{x_i}\right)+2B_i\right].
		\end{equation}
		Here is an algebraic verification. Fix $i$ and write $e=e_i$, $x=x_i$, $B=B_i$,
		\[
		T=x^2+3ex+B,\quad R=x^2-B,\quad S=3x^2+6ex+B.
		\]
		The differential equation for $\wp$ gives
		\[
		y^2=4xT,\quad \wp''(p)=2S,\quad
		C=-\frac{2S}{y},\quad \alpha_i=-\frac{2R}{y}.
		\]
		Let $W_0$ and $D_0$ be the following two polynomials:
		\begin{align*}
			W_0&=S^2-2(x+e)y^2
			=x^4+4ex^3+(12e^2-2B)x^2+4Bex+B^2,\\
			D_0&=-2S^3+6(x+e)Sy^2-y^4
			=2R(x^4+6ex^3+6Bx^2+6Bex+B^2).
		\end{align*}
		By \eqref{add5} and \eqref{add6}, we know $\wp(2p)=W_0/y^2$ and $\wp'(2p)=D_0/y^3$. Therefore, substituting $A_0=-2R$ into \eqref{polynomial N} gives
		\begin{align*}
			y^6N(\alpha_i)
			&=A_0^6+4(\eta_1y^2-2W_0)A_0^4+4D_0A_0^3
			-16D_0(\eta_1y^2+W_0)A_0+16D_0^2\\
			&=512xR^2T^2\bigl((\eta_1+e)(x^2+B)+2Bx\bigr).
		\end{align*}
		The last equality follows by collecting the coefficient of $\eta_1$ and the constant term: they are respectively $512xR^2T^2(x^2+B)$ and $512xR^2T^2(e(x^2+B)+2Bx)$. Using $y^4=16x^2T^2$ and $\alpha_i^2=4R^2/y^2$, we can get \eqref{endpoint factor}.
		
		Notice that the standard Weierstrass-Jacobi formulas \cite[Section 23.6]{DLMF} give
		\begin{equation}\label{endpoint wp Jacobi}
			\sqrt D=2K(m),\quad
			\wp(z)=e_2+\frac{D}{\operatorname{sn}^2(\sqrt D\,z\mid m)},\quad
			\eta_1+e_1=D\frac{E(m)}{K(m)}.
		\end{equation}

We claim that
		\begin{equation}\label{endpoint kappa bound}
			\frac{E(m)}{K(m)}>\sqrt{1-m}=q,\quad 0<\kappa<1-q<m.
		\end{equation}
		To prove \eqref{endpoint kappa bound}, set $a(\theta)=\sqrt{1-m\sin^2\theta}$ and $b(\theta)=\sqrt{1-m\cos^2\theta}$. Pairing $\theta$ with $\pi/2-\theta$ gives
		\[
		E-qK=\frac12\int_0^{\pi/2}(a+b)\left(1-\frac{q}{ab}\right)d\theta>0,
		\]
		where we have used the fact that $$a^2b^2=1-m+m^2\sin^2\theta\cos^2\theta>q^2,\quad\text{for }\;0<\theta<\pi/2.$$ Also, since $E<K$ and $m-(1-q)=q(1-q)>0$, then we can get \eqref{endpoint kappa bound}.
		
		Set
		\begin{equation}\nonumber
			t=t(p)\coloneqq\frac{\wp(p)-e_2}{D}.
		\end{equation}
		The half-period addition formula at $p=1/4$ gives that
		\[
		\wp(1/4)=e_1+\sqrt{(e_1-e_2)(e_1-e_3)}=e_1+Dq.
		\]
		Thus $p\mapsto t(p)$ is strictly decreasing from $(0,1/4)$ onto $(1+q,\infty)$. Moreover,
		\[
		\eta_1+e_2=-D\kappa,\quad \eta_1+e_3=D(m-\kappa),\quad
		B_2=D^2m,\quad B_3=-D^2m(1-m).
		\]
		Consequently, \eqref{endpoint factor} becomes
		\begin{align}
			N(\alpha_2)&=8\alpha_2^2D^2f_2(t),&
			f_2(t)&=2m-\kappa\left(t+\frac mt\right),\label{endpoint f2}\\
			N(\alpha_3)&=8\alpha_3^2D^2f_3(t),&
			f_3(t)&=(m-\kappa)\left(t-m-\frac{m(1-m)}{t-m}\right)-2m(1-m).
			\label{endpoint f3}
		\end{align}
		Note that the factors $8\alpha_2^2D^2$ and $8\alpha_3^2D^2$ are strictly positive. For $t>t_0\coloneqq1+q$,
		\[
		f_2'(t)=-\kappa\left(1-\frac m{t^2}\right)<0,\quad
		f_3'(t)=(m-\kappa)\left(1+\frac{m(1-m)}{(t-m)^2}\right)>0.
		\]
		At the left endpoint,
		\[
		f_2(t_0)=2(m-\kappa)>0,\quad f_3(t_0)=-2\kappa q^2<0,
		\]
		whereas $f_2(t)\to-\infty$ and $f_3(t)\to+\infty$ as $t\to+\infty$. Thus, $f_j(t)$ has a unique simple zero $U_j$ in $(t_0,\infty)$ for $j=2,3$. Solving the corresponding quadratic equations gives exactly $U_2,U_3$ in \eqref{endpoint U2} and \eqref{endpoint U3}, with
		\begin{equation}\label{endpoint signs t}
			N(\alpha_2)>0\Longleftrightarrow t<U_2,\quad
			N(\alpha_3)>0\Longleftrightarrow t>U_3.
		\end{equation}
		
Now we prove $U_3<U_2$. Temporarily replace $\kappa$ by $\kappa_0=1-q$. The two zeros then can be simplified to
		\begin{equation}\label{endpoint comparison roots}
			U_2^0=1+q+\sqrt{2q(1+q)},\quad
			U_3^0=1+q+q\sqrt{2(1+q)}.
		\end{equation}
		Obviously,
		\[
		U_2^0-U_3^0=\sqrt{2(1+q)}(\sqrt q-q)>0.
		\]
		When $t>t_0$, both $f_2,f_3$ strictly decrease as $\kappa$ increases. For $f_3$, we have used the fact that
		\[
		t-m>q(1+q)>\sqrt{m(1-m)},
		\]
		so that $t-m-m(1-m)/(t-m)>0$. Since $\kappa<\kappa_0$, the strict monotonicity in $t$ gives
		\begin{equation}\label{endpoint strong order}
			1+q<U_3<U_3^0<U_2^0<U_2.
		\end{equation}
		This proves that the interval of simultaneous positivity is nonempty. Finally, using \eqref{endpoint signs t}, we get \eqref{endpoint iff} and \eqref{endpoint inverse wp}. 
		%The  signs and simplicity of the zeros in the $p$ variable follow from $t'(p)=\wp'(p)/D<0$.
		
		It remains to prove that $p_-<1/6<p_+$. Define
		\begin{equation}\label{endpoint H}
			H_2(t)\coloneqq\frac{2mt}{t^2+m},\quad
			H_3(t)\coloneqq m-\frac{2m(1-m)(t-m)}{(t-m)^2-m(1-m)},\quad\text{for }t>t_0.
		\end{equation}
		The denominators are positive, and
		\[
		f_2(t)>0\Longleftrightarrow\kappa<H_2(t),\quad
		f_3(t)>0\Longleftrightarrow\kappa<H_3(t).
		\]
		Direct differentiation gives that
		\[
		H_2'(t)=\frac{2m(m-t^2)}{(t^2+m)^2}<0,\quad
		H_3'(t)=\frac{2m(1-m)((t-m)^2+m(1-m))}{((t-m)^2-m(1-m))^2}>0.
		\]
		Also,
		\[
		H_3(U_3^0)=\kappa_0<H_2(U_3^0),\quad
		H_2(U_2^0)=\kappa_0<H_3(U_2^0).
		\]
		It follows that $H_2$ and $H_3$ have a unique intersection $t_*$ in $(t_0,\infty)$ and
		\begin{equation}\label{endpoint intersection}
			U_3^0<t_*<U_2^0,\quad H_2(t_*)=H_3(t_*)>\kappa_0>\kappa.
		\end{equation}
		By a direct calculation, we have
		\begin{equation}\label{endpoint difference}
			H_2(t)-H_3(t)=
			-\frac{m(t^4-4t^3+6mt^2-4m^2t+m^2)}
			{(t^2+m)((t-m)^2-m(1-m))}.
		\end{equation}
		When $p=1/6$, we have $2p=1/2-p$. By the addition formula
		$$\wp\left(z+\frac12\right)=e_1+\frac{(e_1-e_2)(e_1-e_3)}{\wp(z)-e_1}$$
		and
		\eqref{add5}, we know
		\begin{equation}\nonumber
			\frac{\wp(2p)-e_2}{D}=\frac{t-m}{t-1},\quad
			\frac{\wp(2p)-e_2}{D}=\frac{(t^2-m)^2}{4t(t-1)(t-m)}.
		\end{equation}
		Equating these expressions yields that
		\[
		(t^2-m)^2=4t(t-m)^2,
		\]
		or equivalently,
		$$
		t^4-4t^3+6mt^2-4m^2t+m^2=0.
		$$
		Therefore, by \eqref{endpoint difference} and uniqueness, we obtain that $t(1/6)=t_*$. By \eqref{endpoint intersection}, we know both values of $N$ are strictly positive at $p=1/6$. Finally, \eqref{endpoint strong order} and \eqref{endpoint intersection} give $U_3<t(1/6)<U_2$, and thus $p_-<1/6<p_+$. The proof is complete.
	\end{proof}
	In particular, the condition in the first case of Theorem \ref{conditional intersection} is not empty for any rectangular torus. The following corollary is immediate.
	\begin{corollary}\label{corollary final 0 1/4}
		Assume that $0<p<1/4$, then
		\begin{equation}
			\sigma_1\cap\sigma_2=
			\begin{cases}
				\mathcal{Z}\cup\{A_*\}, & \text{if } p\in(p_-(b),p_+(b)); \\
				\mathcal{Z}, & \text{otherwise}.
			\end{cases}
		\end{equation}
	\end{corollary}
	
	Now, we are ready to prove Theorem \ref{main result for 0,1/4}.
	
	\begin{proof}[Proof of Theorem \ref{main result for 0,1/4}] Let $p\in (0, \frac14)$.
		By \eqref{eq: data5} and Theorem \ref{solution = unitary}, we know that the number of one-parameter scaling families of solutions (or equivalently, number of even solutions) equals $$\#((\sigma_1\cap\sigma_2)\setminus\mathcal{Z}).$$
Combining Corollary \ref{corollary final 0 1/4}, we find that the curvature equation \eqref{main problem} has a unique scaling family of solutions if and only if $p\in(p_-(b),p_+(b))$, and each such family has a unique even member. Specifically, when $p\not\in (p_-(b),p_+(b))$, the curvature equation \eqref{main problem} has no solutions.
	\end{proof}
	
	\section{Proof of Theorem \ref{main result general}}\label{section 6}

	In this final section, we consider the general case $\wp(p)\in\R$, $p\not\in E_\tau[2]$ and $\wp'(2p)\neq 0$, and prove Theorem \ref{main result general}. In the sequel, we write $\Delta_k(A)=\Delta_k(A; p,\tau)$ and $\sigma_k=\sigma_k(p,\tau)$ to emphasize their dependence on $p,\tau$.
	
	As mentioned in Section 1, we may assume that
	\begin{equation}\nonumber
		\begin{aligned}
			p&\in\left(0,\frac{1}{4}\right)\cup \left(\frac{1}{4},\frac{1}{2}\right)\cup \left(\frac{1}{2},\frac{1}{2}+\frac{\tau}{4}\right)\cup \left(\frac{1}{2}+\frac{\tau}{4},\frac{1}{2}+\frac{\tau}{2}\right)\\
			&\quad\cup \left(\frac{1}{2}+\frac{\tau}{2},\frac{1}{4}+\frac{\tau}{2}\right)\cup \left(\frac{1}{4}+\frac{\tau}{2},\frac{\tau}{2}\right)\cup \left(\frac{\tau}{2},\frac{\tau}{4}\right)\cup \left(\frac{\tau}{4},0\right).
		\end{aligned}
	\end{equation}
	We will prove that all other cases can be transformed to the basic case $p\in (0, \frac14)\cup(\frac14, \frac12)$.
	
	For simplicity, we set
	\begin{equation}\nonumber
		\mathcal{S}_p(z)=\wp(z+p)+\wp(z-p),\quad \mathcal{H}_p(z)=\zeta(z+p)-\zeta(z-p)-\zeta(2p),
	\end{equation}
	then
	\begin{equation}\nonumber
		\begin{aligned}
			I(z;p,A,\tau)&=\bigg[ 2(\wp(z+p)+\wp(z-p)) + A(\zeta(z+p)-\zeta(z-p))\\
			&\quad +\frac{A^2}{4} - A\zeta(2p) - \frac{2\wp^\prime(2p)}{A} - \wp(2p) \bigg]\\
			&=2\mathcal{S}_p(z)+A\mathcal{H}_p(z)+\frac{A^2}{4} - \frac{2\wp^\prime(2p)}{A} - \wp(2p).
		\end{aligned}
	\end{equation}
	Suppose that $h\in E_\tau[2]$ is a half-period and $\lambda=2h\in\Lambda_\tau$, then we know
	\begin{equation}\nonumber
		\mathcal{S}_{p+h}(z+h)=\wp(z+p+2h)+\wp(z-p)=\mathcal{S}_p(z)
	\end{equation}
	and
	\begin{equation}\nonumber
		\mathcal{H}_{p+h}(z+h)=\zeta(z+p+2h)-\zeta(z-p)-\zeta(2p+2h)=\mathcal{H}_p(z).
	\end{equation}
	Also, by the property of $\wp(z)$, we know
	\begin{equation}\nonumber
		\wp(2p+2h)=\wp(2p)\quad\text{and}\quad \wp'(2p+2h)=\wp'(2p).
	\end{equation}
	Therefore, we have
	\begin{equation}\label{identity for I for translation}
		I(z+h;p+h,A,\tau)=I(z;p,A,\tau).
	\end{equation}
	Suppose that
	\begin{equation}\nonumber
		y''(z)=I(z;p,A,\tau)y(z)
	\end{equation}
	satisfying that
	\begin{equation}\nonumber
		y(z+1)=\varepsilon_1y(z)\quad\text{and}\quad y(z+\tau)+\varepsilon_2y(z)
	\end{equation}
	for some constants $\varepsilon_1,\varepsilon_2\in\C\backslash\{0\}$. Define
	\begin{equation}\nonumber
		\tilde{y}(z)\coloneqq y(z-h),
	\end{equation}
	then it is easy to verify that
	\begin{equation}\nonumber
		\tilde{y}''(z)=I(z-h;p,A,\tau)y(z-h)=I(z;p+h,A,\tau)\tilde{y}(z),
	\end{equation}
	where we have used the identity \eqref{identity for I for translation}. Also, we have
	\begin{equation}\nonumber
		\tilde{y}(z+1)=y(z-h+1)=\varepsilon_1y(z-h)=\varepsilon_1\tilde{y}(z)
	\end{equation}
	and
	\begin{equation}\nonumber
		\tilde{y}(z+\tau)=y(z-h+\tau)=\varepsilon_2y(z-h)=\varepsilon_2\tilde{y}(z).
	\end{equation}
	Therefore, we know
	\begin{equation}\nonumber
		\Delta_j(A;p+h,\tau)=\Delta_j(A;p,\tau),\quad j=1,2,
	\end{equation}
	and hence
	\begin{equation}\label{sigma -- 1st}
		\sigma_j(p+h;\tau)=\sigma_j(p;\tau),\quad j=1,2.
	\end{equation}
	Therefore, we only need to consider the case that
	\begin{equation}
		\begin{aligned}\label{simple p}
			p&\in\left(0,\frac{1}{4}\right)\cup \left(\frac{1}{4},\frac{1}{2}\right)\cup \left(0,\frac{\tau}{4}\right)\cup \left(\frac{\tau}{4},\frac{\tau}{2}\right).
		\end{aligned}
	\end{equation}
	
	Now, consider the transform $p\mapsto -p$. Notice that
	\begin{equation}\nonumber
		\mathcal{S}_{-p}(z)=\wp(z-p)+\wp(z+p)=\mathcal{S}_p(z)
	\end{equation}
	and
	\begin{equation}\nonumber
		\mathcal{H}_{-p}(z)=\zeta(z-p)-\zeta(z+p)+\zeta(2p)=-\mathcal{H}_p(z).
	\end{equation}
	Since $\wp(2p)=\wp(-2p)$ and $\wp'(2p)=-\wp'(-2p)$, it is not hard to see that
	\begin{equation}\nonumber
		\begin{aligned}
			I(z;-p,A,\tau)&=2\mathcal{S}_{-p}(z)+A\mathcal{H}_{-p}(z)+\frac{A^2}{4} - \frac{2\wp^\prime(-2p)}{A} - \wp(-2p)\\
			&=2\mathcal{S}_{p}(z)-A\mathcal{H}_{p}(z)+\frac{A^2}{4} + \frac{2\wp^\prime(2p)}{A} - \wp(2p)\\
			&=I(z;p,-A,\tau).
		\end{aligned}
	\end{equation}
	By a similar argument, we know
	\begin{equation}\nonumber
		\Delta_j(A;-p,\tau)=\Delta_j(-A;p,\tau),\quad j=1,2,
	\end{equation}
	and hence
	\begin{equation}\label{sigma --- 2nd}
		\sigma_j(-p;\tau)=-\sigma_j(p;\tau),\quad j=1,2.
	\end{equation}
	By taking $h=\frac{1}{2}$ and $h=\frac{\tau}{2}$, we know for $j=1,2$ that
	\begin{equation}\label{sigma final 1}
		\sigma_j\left(\frac{1}{2}-p;\tau\right)=-\sigma_j(p;\tau)
	\end{equation}
	and
	\begin{equation}\label{sigma final 2}
		\sigma_j\left(\frac{\tau}{2}-p;\tau\right)=-\sigma_j(p;\tau).
	\end{equation}
	Therefore, without loss of generality, we may assume that
	\begin{equation}\nonumber
		p\in\left(0,\frac{1}{4}\right)\cup \left(0,\frac{\tau}{4}\right).
	\end{equation}
	Recall that when $p\in \left(0,\frac{1}{4}\right)$, we have already figured out the structure of $\sigma_1\cap\sigma_2$, see Theorem \ref{conditional intersection}. Hence, we now focus on the case $p\in \left(0,\frac{\tau}{4}\right)$.
	
	Fix $\tau=ib$, $b>0$. Denote $p=iv$, $v\in(0,b/4)$. Let
	\begin{equation}\nonumber
		\hat{\tau}=-\frac{1}{\tau}=\frac{i}{b},\quad\hat{p}=\frac{p}{\tau}=\frac{v}{b}\in\left(0,\frac{1}{4}\right).
	\end{equation}
  Recall the modular property of $\wp(u;\tau)$ and $\zeta(u;\tau)$:
\begin{equation}\nonumber
	\wp(u;\hat{\tau})=\tau^2\wp(\tau u;\tau),\quad 
	\wp'(u;\hat{\tau})=\tau^3\wp'(\tau u;\tau),
\end{equation}
and
\begin{equation}\nonumber
	\zeta(u;\hat{\tau})=\tau\zeta(\tau u;\tau).
\end{equation}
Now, let $z=\tau u$ and $\hat{A}=\tau A$. By the above modular properties, we have
\begin{equation}\nonumber
	\wp(\tau u\pm p;\tau)=\tau^{-2}\wp(u\pm\hat{p};\hat{\tau}),
\end{equation}
\begin{equation}\nonumber
	\zeta(\tau u\pm p;\tau)=\tau^{-1}\zeta(u\pm \hat{p};\hat{\tau}).
\end{equation}
Moreover,
\begin{equation}\nonumber
	\wp(2p;\tau)=\tau^{-2}\wp(2\hat{p};\hat{\tau}),\quad \wp'(2p;\tau)=\tau^{-3}\wp'(2\hat{p};\hat{\tau}),\quad\zeta(2p;\tau)=\tau^{-1}\zeta(2\hat{p};\hat{\tau}).
\end{equation}
For simplicity, we write
\begin{equation}\nonumber
	\hat{\mathcal{S}}_{\hat{p}}(u)=\wp(u+\hat{p};\hat{\tau})+\wp(u-\hat{p};\hat{\tau}),\quad \hat{\mathcal{H}}_{\hat{p}}(u)=\zeta(u+\hat{p};\hat{\tau})-\zeta(u-\hat{p};\hat{\tau})-\zeta(2\hat{p};\hat{\tau}),
\end{equation}
then
\begin{equation}\nonumber
	\begin{aligned}
		\tau^2I(\tau u;p,A,\tau)&=\tau^2\bigg[ 2(\wp(\tau u+p;\tau)+\wp(\tau u-p;\tau)) + A(\zeta(\tau u+p;\tau)-\zeta(\tau u-p;\tau))\\
		&\quad +\frac{A^2}{4} - A\zeta(2p;\tau) - \frac{2\wp^\prime(2p;\tau)}{A} - \wp(2p;\tau) \bigg]\\
		&=2\hat{\mathcal{S}}_{\hat{p}}(u)+\tau A\hat{\mathcal{H}}_{\hat{p}}(u)+\frac{\tau^2A^2}{4}-\frac{2\wp'(2\hat{p};\hat{\tau})}{\tau A}-\wp(2\hat{p};\hat{\tau})\\
		&=2\hat{\mathcal{S}}_{\hat{p}}(u)+ \hat{A}\hat{\mathcal{H}}_{\hat{p}}(u)+\frac{\hat{A}^2}{4}-\frac{2\wp'(2\hat{p};\hat{\tau})}{\hat{A}}-\wp(2\hat{p};\hat{\tau})\\
		&=I(u;\hat{p},\hat{A},\hat{\tau}).
	\end{aligned}
\end{equation}
Suppose that
\begin{equation}\nonumber
	y''(u)=I(u;p,A,\tau)y(u)
\end{equation}
satisfying that
\begin{equation}\nonumber
	y(u+1)=\varepsilon_1y(u)\quad\text{and}\quad y(u+\tau)=\varepsilon_2y(u)
\end{equation}
for some constants $\varepsilon_1,\varepsilon_2\in\C\backslash\{0\}$. Define
\begin{equation}\nonumber
	\hat{y}(u)\coloneqq y(\tau u),
\end{equation}
then it is easy to verify that
\begin{equation}\nonumber
	\hat{y}''(u)=\tau^2 y''(\tau u)=\tau^2 I(\tau u;p,A,\tau)y(\tau u)=I(u;\hat{p},\hat{A},\hat{\tau})\hat{y}(u).
\end{equation}
Moreover, we have
\begin{equation}\nonumber
	\hat{y}(u+1)=y(\tau(u+1))=y(\tau u +\tau)=\varepsilon_2y(\tau u)=\varepsilon_2 \hat{y}(u)
\end{equation}
and
\begin{equation}\nonumber
	\hat{y}(u+\hat{\tau})=y(\tau(u+\hat{\tau}))=y(\tau u -1)=\varepsilon_1^{-1}y(\tau u)=\varepsilon_1^{-1} \hat{y}(u).
\end{equation}
These identities actually imply that
\begin{equation}\nonumber
	\Delta_1(A;p,\tau)=\Delta_2(\hat{A};\hat{p},\hat{\tau}),\quad \Delta_2(A;p,\tau)=\Delta_1(\hat{A};\hat{p},\hat{\tau}),
\end{equation}
hence
\begin{equation}
	\sigma_1(p;\tau)=\frac{1}{\tau}\sigma_2(\hat{p},\hat{\tau}),\quad \sigma_2(p;\tau)=\frac{1}{\tau}\sigma_1(\hat{p},\hat{\tau}).
\end{equation}
Therefore, we know
\begin{equation}\label{relation of stability set}
	\sigma_1(p;\tau)\cap \sigma_2(p;\tau)=\frac{1}{\tau}(\sigma_1(\hat{p},\hat{\tau})\cap \sigma_2(\hat{p},\hat{\tau})).
\end{equation}
For $p\in(0,\frac{\tau}{4})$ and $\tau=ib$, $b>0$, write
\begin{equation}\nonumber
	\hat{\alpha}_1=\frac{\alpha_2(\hat{p};\hat{\tau})}{\tau},\quad \hat{\alpha}_2=\frac{\alpha_1(\hat{p};\hat{\tau})}{\tau},\quad \hat{\alpha}_3=\frac{\alpha_3(\hat{p};\hat{\tau})}{\tau},
\end{equation}
\begin{equation}\nonumber
	\hat{\beta}_1=\frac{\beta_3(\hat{p};\hat{\tau})}{\tau},\quad \hat{\beta}_2=\frac{\beta_2(\hat{p};\hat{\tau})}{\tau},\quad \hat{\beta}_3=\frac{\beta_1(\hat{p};\hat{\tau})}{\tau},
\end{equation}
and $\hat{\mathcal{Z}}\coloneqq\{\hat{\beta}_1,\hat{\beta}_2,\hat{\beta}_3,\pm\hat{\alpha}_1,\pm\hat{\alpha}_2,\pm\hat{\alpha}_3\}$. 
Now, combining Theorem \ref{conditional intersection}, Theorem \ref{endpoint positivity theorem}, and \eqref{relation of stability set}, we finally get the following result:

\begin{theorem}\label{conditional intersection vertical}
	Assume that $p\in(0,\tau/4)$, then
	\begin{equation}
		\sigma_1\cap\sigma_2=
		\begin{cases}
			\hat{\mathcal{Z}}\cup\left\{\frac{{A}_*(\hat{p};\hat{\tau})}{\tau}\right\}, & \text{if }\hat{p}=\frac{p}{\tau}\in\left(p_-(\frac{1}{b}),p_+(\frac{1}{b})\right); \\
			\hat{\mathcal{Z}}, & \text{otherwise}.
		\end{cases}
	\end{equation}
\end{theorem}
Now we are ready to establish Theorem \ref{main result general}.
\begin{proof}[Proof of Theorem \ref{main result general}]
	By Corollary \ref{corollary final 0 1/4} and Theorem \ref{conditional intersection vertical}, we have already figured out the existence of solutions to the curvature equation \eqref{main problem} when $p\in(0,1/4)\cup(0,\tau/4)$. The other cases follow by the translation and reflection laws \eqref{sigma -- 1st}, \eqref{sigma --- 2nd}, \eqref{sigma final 1}, and \eqref{sigma final 2}.
\end{proof}

\section*{Acknowledgements} Z. Chen is supported by National Key R\&D Program of China (No. 2023YFA1010002) and NSFC (No. 12222109).

\section*{Declaration of generative AI in the manuscript preparation process}
During the preparation of this work the authors used OpenAI model as an auxiliary tool. After using this tool, the authors reviewed and edited the content as needed and take full responsibility for the content of the published article.

	\bibliographystyle{amsplain}

\end{document}